\documentclass[oneside,english]{amsart}
\usepackage[T1]{fontenc}
\usepackage[latin9]{inputenc}
\usepackage{geometry}
\usepackage{color}
\usepackage{units}
\usepackage{amstext}
\usepackage{amsthm}
\usepackage{amssymb}
\usepackage{stmaryrd}
\usepackage{marginnote}
\usepackage{graphicx}
\usepackage{breqn}
\usepackage[foot]{amsaddr}

\makeatletter
\numberwithin{equation}{section}
\numberwithin{figure}{section}
\theoremstyle{plain}
\newtheorem{thm}{\protect\theoremname}[section]
  \theoremstyle{plain}
  \newtheorem{lem}[thm]{\protect\lemmaname}
  \theoremstyle{plain}
  \newtheorem{prop}[thm]{\protect\propositionname}
  \theoremstyle{remark}
  \newtheorem*{rem*}{\protect\remarkname}
  \theoremstyle{remark}
  \newtheorem{rem}[thm]{\protect\remarkname}

\usepackage{babel}
\usepackage{amsmath}

\makeatother

\usepackage{babel}
  \providecommand{\lemmaname}{Lemma}
  \providecommand{\propositionname}{Proposition}
  \providecommand{\remarkname}{Remark}
\providecommand{\theoremname}{Theorem}

\graphicspath{{Dibujos/}}

\begin{document}

\title{Symbolic dynamics in the restricted elliptic isosceles 
three body problem}

\author{Marcel Guardia\textsuperscript{a,b}}
\email{marcel.guardia@upc.edu}
\author{Jaime Paradela\textsuperscript{a,*}}\thanks{* Corresponding author}
\email{jaime.paradela@upc.edu}
\author{Tere M. Seara\textsuperscript{a,b}}
\email{tere.m-seara@upc.edu}
\author{Claudio Vidal\textsuperscript{c}}
\email{clvidal@ubiobio.cl}

\address[A]{Departament de Matem{\`a}tiques \&IMTECH, Universitat Polit{\`e}cnica de Catalunya\\
Diagonal 647, 08028 Barcelona, Spain}

\address[B]{Centre de Recerca Matem{\`a}tica, Campus de Bellaterra\\
Edifici C, 08193 Barcelona, Spain}

\address[C]{Departamento de Matem{\'a}ticas, Facultad de Ciencias, Universidad del B{\'i}o-B{\'i}o\\
Casilla 5C, VIII Regi{\'o}n, Concepci{\'o}n, Chile}

\begin{abstract}
{\normalsize{}The restricted elliptic isosceles three body problem
(REI3BP) models the motion of a massless body under the influence
of the Newtonian gravitational force caused by two other bodies called
the primaries. The primaries of masses $m_{1}=m_{2}$ move along a
degenerate Keplerian elliptic collision orbit (on a line) under their 
gravitational
attraction, whereas the third, massless particle, moves on the plane
perpendicular to their line of motion and passing through the center
of mass of the primaries. By symmetry, the component of the angular
momentum $G$ of the massless particle along the direction of the
line of the primaries is conserved. }{\normalsize \par}

{\normalsize{}We show the existence of symbolic dynamics in the REI3BP
for large $G$ by building a Smale horseshoe on a certain subset of
the phase space. As a consequence we deduce that the REI3BP possesses 
oscillatory
motions, namely orbits which leave every bounded region but return
infinitely often to some fixed bounded region. The proof relies on
the existence of transversal homoclinic connections associated to an invariant
manifold at infinity. Since the distance between the stable and unstable
manifolds of infinity is exponentially small, Melnikov theory does
not apply.}{\normalsize \par}
\end{abstract}

\maketitle

\section{\label{sec:Introduction}Introduction}

The restricted three body problem studies the motion of three bodies,
one of them massless, under Newtonian gravitational force. The massless
body does not exert any force on the other two, the
primaries, and move therefore according to Kepler laws. As a particular
case, in the restricted elliptic isosceles three body problem (REI3BP),
the primaries move along a degenerate ellipse and the third (massless)
body moves on the perpendicular plane to their line of motion passing through 
their center of mass, which is invariant. In this configuration the primaries
collide, but since it is a Keplerian motion its collisions can be
regularized. In a coordinate system with origin at the center of mass of the primaries, the position of the primaries is given by
\begin{equation}
q_1\left(t\right)=\frac{\rho\left(t\right)}{2}\left(0,0,1\right)\qquad q_2\left(t\right)=\frac{\rho\left(t\right)}{2}\left(0,0,-1\right),
\end{equation} 
where 
\begin{equation}
\rho\left(t\right)=1-\cos E\left(t\right)
\end{equation}
and the eccentric anomaly $E\left(t\right)$ satisfies
\begin{equation}
t=E-\sin E.
\end{equation}
Introducing polar coordinates $\left(r,y,\alpha,G\right)$ in the plane of motion of the third body, where $\left(y,G\right)$ denote the conjugated momenta to $\left(r,\alpha\right)$ the REI3BP is Hamiltonian with respect to  
\begin{equation}\label{eq:Hamiltoniano_original}
H\left(r,y,G,t\right)=\frac{y^{2}}{2}+\frac{G^{2}}{r^{2}}-\frac{1}{\sqrt{r^{2}+\frac{\rho^{2}\left(t\right)}{4}}}.
\end{equation}
It is inmediate to check that $G$ is a conserved quantity so the REI3BP is a system of $1+1/2$ degrees of freedom. We fix $G\neq0$ in order to avoid triple collisions. 

In  \cite{Brandao2008} authors the  study the existence of symmetric periodic 
solutions of the Hamiltonian system associated to 
\eqref{eq:Hamiltoniano_original}.
In the present paper we prove the existence of chaotic dynamics in the REI3BP for large values of the angular momentum $G$, by building a Smale horseshoe with infinitely many symbols on a certain subset of the phase space. To build this horseshoe we first prove that the stable and unstable manifold associated to a certain invariant manifold intersect transversally, giving rise to homoclinic connections to the invariant manifold.  

As a consequence, from the way the horseshoe is built, we deduce the existence  
of different types of orbits  of the REI3BP according to their behavior as 
$t\to\pm\infty$. In particular, the existence of infinitely many 
periodic orbits of arbitrary large period is obtained. A complete 
classification of the orbits
of the three body problem according to their final motion was already
established by Chazy in 1922 (see \cite{Arnold1988}). For the restricted
three body problem (either planar or spatial, circular or elliptic)
the possibilities reduce to four:
\begin{itemize}
\item $H^{\pm}\text{(hyperbolic)}:\left\Vert r\left(t\right)\right\Vert \to\infty\text{ and }\left\Vert \dot{r}\left(t\right)\right\Vert \to c>0\text{ as }t\to\pm\infty.$
\item $P^{\pm}\text{(parabolic)}:\left\Vert r\left(t\right)\right\Vert \to\infty\text{ and }\left\Vert \dot{r}\left(t\right)\right\Vert \to0\text{ as }t\to\pm\infty.$
\item $B^{\pm}\text{(bounded)}:\limsup_{t\to\pm\infty}\left\Vert r\left(t\right)\right\Vert <\infty.$
\item $OS^{\pm}\text{(oscillatory)}:\limsup_{t\to\pm\infty}\left\Vert r\left(t\right)\right\Vert =\infty\text{ and }\liminf_{t\to\pm\infty}\left\Vert r\left(t\right)\right\Vert <\infty.$ 
\end{itemize}
Examples of hyperbolic, parabolic
and bounded motions were already known by Chazy (in particular they are present in the two body problem). However, no examples
of oscillatory motions  were known until Sitnikov 
\cite{Sitnikov1960}
proved their existence on a certain symmetric configuration of the
spatial restricted three body problem, now called \textit{the Sitnikov
example}. We shall prove that any past-future combination of the four possible
final motions exists in the REI3BP. 

The connection between chaotic dynamics and the existence of different types of 
final motions was first devised by Moser \cite{moser1973stable}, who gave a new 
proof of the existence of oscillatory motions in the Sitnikov model. Moser's 
approach relying on the connection between final motions,
transversal homoclinic points and symbolic dynamics has been successfully
extended to provide more examples of these motions \cite{SimoL80, Llibre1980, Moeckel84, Moeckel07,  GorodetskiK12, Guardia2016}. When dealing with
perturbations of integrable systems the classical strategy for showing
the existence of transversal intersections between the invariant manifolds
is to find non-degenerate zeros of the Melnikov function, which gives
an asymptotic expression for the distance between them. However, when
considering fast non-autonomous perturbations, the Melnikov function
is exponentially small with respect to the perturbative parameter
and the validity of Melnikov theory is not justified. This difficulty
can be solved when the system in consideration has two perturbative
parameters and an exponentially smallness condition between them is
assumed. This was the approach in \cite{SimoL80}, where the existence
of oscillatory motions in the restricted planar circular three body
problem (RPC3BP) was shown for values of the mass ratio exponentially
small compared to the value of the inverse of the Jacobi constant. 

The study of the existence of intersections between invariant manifolds
for fast non-autonomous perturbations without assuming smallness conditions
on extra parameters requires showing that the distance between invariant
manifolds is indeed exponentially small. This problem, now known as
\textit{exponentially small splitting of separatrices}, has drawn
remarkable attention in the past decades, but, due to its difficulty
most of the available results concern concrete models 
\cite{Holmes1988,Delshams1992,Gelfreich00,Guardia2010, Gaivao2011}
or in general systems under very restrictive hypothesis to be applicable
to problems in Celestial Mechanics 
\cite{Delshams1997,Gelfreich1997,Treshev97,Guardia2012,Baldoma2012,Baldoma2004a, Gaivao2012}. Following these ideas, \cite{Guardia2016} proves the 
transversality of certain invariant manifolds of the RPC3BP for 
any mass ratio and large Jacobi constants, extending the result in 
\cite{Llibre1980} of existence of oscillatory motions to any mass ratio.

Following the same approach in \cite{Guardia2016},  the present paper proves 
the exponentially small splitting
of separatrices in a real problem arising from Celestial Mechanics,
the aforementioned REI3BP, under the only assumption of large angular
momentum $G$. 
It is worth pointing out that the Hamiltonian \eqref{eq:Hamiltoniano_original} is, in general, far from being integrable. However, we will see in Section \ref{sec:Description_of_the_proof} that for orbits with large angular momentum $G$, the Hamiltonian \eqref{eq:Hamiltoniano_original} can be considered as a fast non-autonomous perturbation of the two body problem, which is integrable.

From our result we deduce the existence of transverse
homoclinic connections and we are able to build a Smale horseshoe
on a certain subset which is close to the homoclinc points. This result
completes the previous work \cite{Brandao2017}, where the existence
of symbolic dynamics in the EIR3BP was investigated for large values
of $G$ using numerical techniques for analyzing the exponentially
small splitting of separatrices.

The main result of the present paper, which gives the existence
of chaotic dynamics in the REI3BP, is the following.
\begin{thm}
\label{thm:Main_theorem}Denote by  $\psi$ the Poincar{\'e} map induced by the flow of the Hamiltonian \eqref{eq:Hamiltoniano_original} on the section
 $\varSigma_+=\left\{\left(r,y,t\right)\in\mathbb{R}_+\times\mathbb{R}\times\mathbb{T}:y=0,\ \dot{y}>0\right\}$. 
Then, there exists $0<G^{*}<\infty$  such that for $G>G^{*}$ there exists an invariant
set $S\subset\varSigma_+$ such that the dynamics of $\psi:S\to S$
is topologically conjugated to the shift
\[
\begin{split}\sigma:\mathbb{N}^{\mathbb{Z}} & \to\mathbb{N^{Z}}\\
\left\{ a_{n}\right\} _{n\in\mathbb{\mathbb{Z}}} & \mapsto\left\{ a_{n-1}\right\} _{n\in\mathbb{Z}}
\end{split}
\]

Namely $\psi$ has a Smale horseshoe of infinite symbols.
\end{thm}

An immediate consequence of Theorem \ref{thm:Main_theorem} is the existence of infinitely many periodic orbits in the system associated to Hamiltonian \eqref{eq:Hamiltoniano_original}. Moreover,  from the way the Smale horseshoe of Theorem \ref{thm:Main_theorem}
is built, we obtain the second main result (see Section \ref{sec:Description_of_the_proof}
for a detailed exposition of this connection). 
\begin{thm}
\label{thm:Main_Theorem_b}Denote by $X^{+}$ (respectively $Y^{-})$
either $H^{+},P^{+},B^{+}$ or $OS^{+}$ (respectively $H^{-},P^{-},B^{-}$
or $OS^{-})$. Then, there exists $G^{*}<\infty$ such that if $G>G^{*}$
we have 
\[
X^{+}\cap Y^{-}\neq\emptyset
\]
for all possible combinations of $X^{+}$ and $Y^{-}$. In particular,
the Hamiltonian system \eqref{eq:Hamiltoniano_original} posses oscillatory
orbits, that is, orbits such that 
\[
\limsup_{t\rightarrow\pm\infty}\left|r\left(t\right)\right|=\infty\quad\text{and}\quad\liminf_{t\rightarrow\pm\infty}\left|r\left(t\right)\right|<\infty.
\]
\end{thm}

As commented above, the proof of Theorem \ref{thm:Main_theorem} relies
on two main ingredients: establishing the existence of transversal
intersections between the invariant manifolds $\mathcal{W}_{\infty}^{u,s}$
associated to a periodic orbit at infinity and showing the existence
of a Smale horseshoe on a certain subset close to the homoclinc points.
The latter follows from the arguments presented in \cite{moser1973stable}
without significant modifications. These arguments are sketched in
Section \ref{sec:Description_of_the_proof} for the sake of self-completeness. 

For the analysis of the splitting of the invariant manifolds, we use
the fact that $\mathcal{W}_{\infty}^{u,s}$ are Lagrangian submanifolds
so they can be parametrized as graphs which satisfy the Hamilton-Jacobi
equation associated to $H$. 
Then, we study  solutions to this equation
in a suitable complex domain to get exponentially small  asymptotics for
the distance between $\mathcal{W}_{\infty}^{s}$ and $\mathcal{W}_{\infty}^{u}$.
In order to obtain the appropiate exponent these parameterizations must be 
analyzed in a neighbourhood
$\mathcal{O}\left(G^{-3}\right)$ of the singularities of the unperturbed
homoclinic $\left(G\to\infty\right).$

The document is organized as follows. In Section \ref{sec:Description_of_the_proof}
we introduce the invariant manifolds at infinity and discuss the proofs
and connection between Theorem \ref{thm:Main_theorem} and Theorem
\ref{thm:Main_Theorem_b}. In particular, from Theorem \ref{thm:Main Theorem 2-1},
which claims the existence of transverse intersections of the infinity
manifolds, we build a Smale horseshoe that is then used to show the
existence of any past-future combination of final motions. The rest
of the paper is devoted to the proof of Theorem \ref{thm:Main Theorem 2-1}.
We discuss the integrable system $(G\to\infty)$ and its homoclinic
manifold in Section \ref{sec:The-integrable-system}. Section \ref{sec:Parametrizations-of-the}
is devoted to rewrite the problem of existence of the infinity manifolds
as a fixed point equation. We solve this equation and bound the solution
in a suitable complex domain in Section \ref{sec:Existence-of-the}.
In Section \ref{sec:The-difference-between} we show that the distance
between the invariant manifolds is given, up to first order, by the
Melnikov function and then we compute its asymptotic expansion for
large $G$ in Section \ref{sec:Computation-of-the}. 

\section{\label{sec:Description_of_the_proof}Description of the proof of
theorems \ref{thm:Main_theorem} and \ref{thm:Main_Theorem_b}}

We notice from the Hamiltonian \eqref{eq:Hamiltoniano_original} that the angular momentum $G$  is a conserved quantity. Therefore, we apply the conformally symplectic change of variables 
\[
r=G^{2}\tilde{r},\quad y=G^{-1}\tilde{y},\quad t=G^{3}s,
\]
to the equations of motion associated to the Hamiltonan \eqref{eq:Hamiltoniano_original}
to obtain a new system which is also Hamiltonian with respect to the
scaled Hamiltonian. 
\begin{equation}
\begin{split}\tilde{H}\left(\tilde{r},\tilde{y},s;G\right) & =G^{2}H\left(G^{2}\tilde{r},G^{-1}\tilde{y},G^{3}s\right)\\
 & =\frac{\tilde{y}^{2}}{2}+\frac{1}{\tilde{r}^{2}}-\frac{1}{\tilde{r}}+U\left(\tilde{r},G^{3}s\right)
\end{split}
\label{eq:rescaled_hamiltonian}
\end{equation}
where 
\begin{equation}
U\left(\tilde{r},G^{3}s\right)=\frac{1}{\tilde{r}}-\frac{1}{\sqrt{\tilde{r}^{2}+\nicefrac{\rho^2\left(G^{3}s\right)}{4G^{4}}}}=\frac{\rho^{2}\left(G^{3}s\right)}{8G^{4}\tilde{r}^{3}}\left(1+\mathcal{O}\left(\frac{1}{\tilde{r}^{2}G^{4}}\right)\right).\label{eq:potential_main_order_term}
\end{equation}
Observe that, for $G$ large, the system associated to the Hamiltonian
\eqref{eq:rescaled_hamiltonian} can be studied as a fast and small non-autonomous
perturbation of the Kepler two-body problem. Adding time $t$ as a phase variable, which we now denote by $\xi$, we see from the equations of
motion associated to the Hamiltonian \eqref{eq:rescaled_hamiltonian}
\begin{equation}
\begin{split}\frac{\text{d}\tilde{r}}{\text{d}s} & =\tilde{y}\\
\frac{\text{d}\tilde{y}}{\text{d}s} & =\frac{1}{\tilde{r}^{3}}-\frac{1}{\tilde{r}^{2}}-\partial_{\tilde{r}}U\\
\frac{\text{d}\xi}{\text{d}s} & =G^{3},
\end{split}
\label{eq:eqs of motion}
\end{equation}
that $\varLambda=\left\{ \left(\tilde r,\tilde y,\xi\right)=\left(\infty,0,\xi\right):\xi\in\mathbb{\mathbb{T}}\right\} $
is a parabolic periodic orbit, which
we will call infinity. 

Denoting by 
$\phi_{s}=(\phi^{\tilde r}_{s},\phi^{\tilde y}_{s},\phi^\xi_{s})$ the flow of the
system \eqref{eq:eqs of motion}, we define the stable and unstable
manifolds of infinity as 
\begin{equation}
\begin{split}\mathcal{W}_{\infty}^{s} & =
\left\{ 
\left(\tilde{r},\tilde{y},\xi\right)\in\mathbb{R}_+\times\mathbb{R}\times\mathbb
{T}:\lim_{s\rightarrow+\infty}\phi_{s}^{\tilde 
r}\left(\tilde{r},\tilde{y},\xi\right)= \infty, \ 
\lim_{s\rightarrow+\infty}\phi_{s}^{\tilde 
y}\left(\tilde{r},\tilde{y},\xi\right)= 0\right\} \\
\mathcal{W}_{\infty}^{u} & =\left\{ \left(\tilde{r},\tilde{y},\xi\right)\in\mathbb{R}_+\times\mathbb{R}\times\mathbb{T}:\lim_{s\rightarrow-\infty}\phi_{s}^{\tilde r}\left(\tilde{r},\tilde{y},\xi\right)=\infty, \ \lim_{s\rightarrow-\infty}\phi_{s}^{\tilde y}\left(\tilde{r},\tilde{y},\xi\right)=0\right\} .
\end{split}
\label{eq:infinity_manifolds}
\end{equation}
The usual way to study the dynamics near infinity is to use McGehee
coordinates $r=2x^{-2}$ which map neighbourhoods of infinity into
bounded domains containing the origin. In particular, the 
periodic orbit $\varLambda$ corresponds to the periodic orbit 
$\left\{ \left(x,y,\xi\right)=\left(0,0,\xi\right)\colon\xi\in\mathbb{T}\right\} $
in McGehee coordinates. This transformation was used in \cite{McGehee1973}
to show that $\mathcal{W}_{\infty}^{u,s}$ exist and are analytic
submanifolds except at infinity, where only $C^{\infty}$ regularity
is proven (see \cite{Baldoma2020a,Baldoma2020b} for more general results).
However, in the present work we prefer to stick to the original variables
since the symplectic form is non canonical in McGehee coordinates.

For $G\to\infty$ the system is integrable since $U\to0$ and therefore $\mathcal{W}_{\infty}^{s}$
and $\mathcal{W}_{\infty}^{u}$ coincide along a two dimensional homoclinic
manifold which is foliated by Keplerian parabolic orbits. Hence, it can
be parametrized by the time section $\xi$ and a suitable time parametrization
$\left(\tilde{r}_{h}\left(v\right),\tilde{y}_{h}\left(v\right)\right)$
of the parabolic orbit. We denote the parametrization of this invariant
manifold as
\begin{equation}
\tilde{z}_{h}\left(v,\xi\right)=\left(\tilde{r}_{h}\left(v\right),\tilde{y}_{h}\left(v\right),\xi\right)\qquad\text{where\qquad}\left(v,\xi\right)\in\mathbb{R}\times\mathbb{T},\label{eq:unperturbed_homoclinic_first}
\end{equation}
and fix the origin of $v$ such that $\tilde{y}_{h}\left(0\right)=0$,
which makes the homoclinic orbit symmetric under the map $v\to-v.$
Some properties of this parametrization are discussed in Section \ref{sec:The-integrable-system}. 

We will prove that in the full problem \eqref{eq:rescaled_hamiltonian}, this two dimensional homoclinic manifold breaks down for $1\ll G<\infty$, and $\mathcal{W}_{\infty}^{s},\mathcal{W}_{\infty}^{u}$
do not longer coincide. 
In order to measure the distance between the
invariant manifolds we introduce the Poincar{\'e} stroboscopic map 
\begin{eqnarray}
\mathcal{P}_{\xi_{0}}:\left\{ \xi=\xi_{0}\right\}  & \rightarrow & \left\{ \xi=\xi_{0}+2\pi\right\} \label{eq:Poincare_map}\\
\left(\tilde{r},\tilde{y}\right) & \mapsto & \mathcal{P}_{\xi_{0}}\left(\tilde{r},\tilde{y}\right)\nonumber 
\end{eqnarray}
so $\mathcal{W}_{\infty}^{s,u}\cap\left\{\xi=\xi_0\right\}$ become invariant curves $\gamma^{s,u}$
(see Figure \ref{figure-1}).

\begin{figure}
\begin{center}
\includegraphics[scale=0.5]{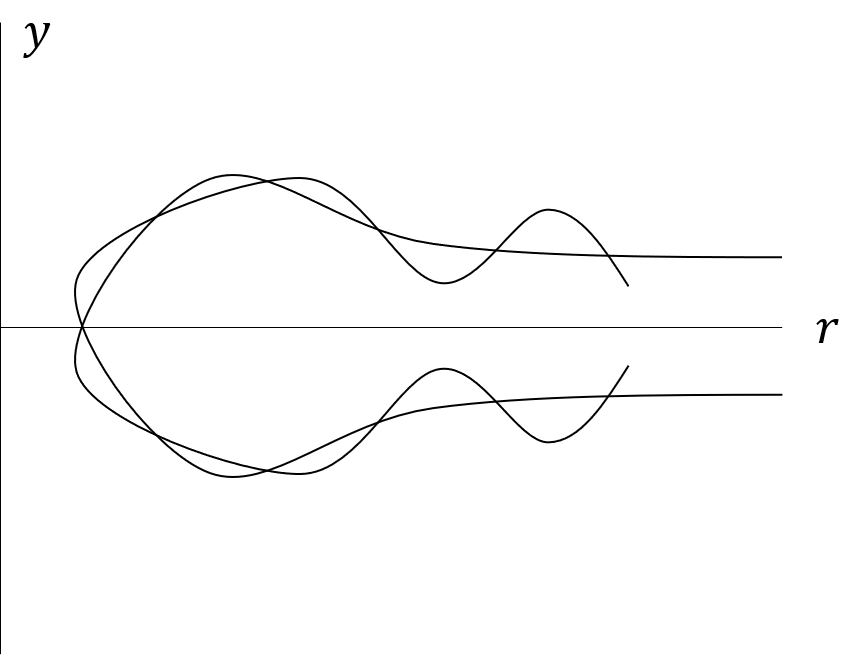}
\caption{Stable and unstable invariant manifolds of infinity for the Poincar{\'e}
map $\mathcal{P}_{\xi_{0}}$ in (\ref{eq:Poincare_map}). }
\label{figure-1} 
\end{center}
\end{figure}

Then, for $y>0$,  considering a parametrization of $\gamma^{s,u}$ of the form
\begin{equation}
\begin{array}{rl}
\tilde{r} & =\tilde{r}_{h}\left(v\right)\\
\tilde{y} & =Y_{\xi_{0}}^{s,u}\left(v\right)
\end{array}\label{eq:parametrization inv curves}
\end{equation}
 where $\tilde{r}_{h}\left(v\right)$ is the parametrization of the
unperturbed homoclinic \eqref{eq:unperturbed_homoclinic_first},
we observe that to measure the distance between the invariant manifolds
along a suitable section $v=v^*$ it suffices to measure
the difference between the functions $Y_{\xi_{0}}^{s,u}$. The following
theorem is one of the two main ingredients needed for the proof of
Theorem \ref{thm:Main_theorem}.
\begin{thm}
\label{thm:Main Theorem 2-1}Let $\mathcal{W}_{\infty}^{s}$ and $\mathcal{W}_{\infty}^{u}$
be the infinity manifolds associated to the  periodic orbit 
$\varLambda$
and $\gamma^{s,u}$ the corresponding curves of the map $\mathcal{P}_{\xi_{0}}.$
Then, for $G$ large enough,
\begin{itemize}
\item[(i)] 
The curves $\gamma^{s,u}$ exist and have a parametrization
of the form \eqref{eq:parametrization inv curves},

\item[(ii)] 
If we fix a section $\tilde{r}=\tilde{r}\left(v^{*}\right)$
the distance $d$ between these curves along this section is given
by 
\begin{equation}
d=\frac{J_{1}\left(1\right)\sqrt{2\pi}}{\tilde{y}_{h}\left(v^{*}\right)}G^{
\nicefrac{1}{2}}e^{-\frac{G^{3}}{3}}\sin\left(\xi_{0}-G^{3}v^{*}\right)+E,
\qquad\left|E\right|\leq CG^{-1/2}e^{-\frac{G^{3}}{3}},\label{eq:distance 
formula}
\end{equation}
where $J_{1}$ is the first Bessel function of first kind and $\tilde{y}_{h}$
correspond to the $\tilde{y}$ component of the unperturbed homoclinic and $C>0$ is a constant independent of $G$. 

\item[(iii)] There exist (at least) two transverse homoclinic connections
to the  periodic orbit  $\varLambda.$
\end{itemize}
\end{thm}
Item $(iii)$ is a direct consequence of Item $(ii)$. Indeed,
since 
\[
J_{1}\left(1\right)\sim0.44051\neq0
\]
we observe that formula \eqref{eq:distance formula} in Theorem \ref{thm:Main 
Theorem 2-1},
implies that the zeros of the distance are given, up to first order, by the
zeros of the function $\sin\left(\xi_{0}-G^{3}v^{*}\right)$. Therefore, transversal
intersections of the invariant curves $\gamma^{s,u}$ will occur for
values of $\xi_{0}-G^{3}v^{*}$ located in a neighbourhood $\mathcal{O}\left(G^{-1}\right)$
of the points $\xi_{0}-G^{3}v^{*}=0,\pi.$ These transversal intersections
give rise to two homoclinic connections to the invariant manifold
$\varLambda$ as stated in the third item of Theorem \ref{thm:Main Theorem 2-1}.

Observe that the distance between the invariant manifolds is exponentially small with respect to $G.$
As usually happens in exponentially small splitting of separatrices
phenomena, the smaller the period of the fast perturbation (in our case $2\pi/G^3$), the smaller
the distance between the manifolds (see \cite{Neishtadt1984}).

\subsection{Symbolic dynamics and oscillatory orbits}

Once Theorem \ref{thm:Main Theorem 2-1} is proven, the existence of chaotic dynamics
is obtained following the techniques introduced
in \cite{moser1973stable}. For that we define the section  
\begin{equation}
\varSigma_{+}=\left\{\left(\tilde{r},\tilde{y},\xi\right)\in \mathbb{R}_+\times\mathbb{R}\times\mathbb{T}:\tilde{y}=0, \dot{\tilde{y}}>0\right\}
\end{equation}
 and use coordinates $\left(\tilde{r}_0,\xi_0\right)$ for this section. Then, we define the Poincar{\'{e}} map
\begin{equation}
\begin{split}\psi:\varSigma_{+} & \rightarrow\varSigma_{+}\\
\left(\tilde{r}_{0},\xi_{0}\right) & \mapsto\left(\tilde{r}_{1},\xi_{1}\right)
\end{split}
\label{eq:map for the horseshoe}
\end{equation}
where $\xi_{1}=\xi_0+G^3 s$, and $s>0$ is the first time in 
which $\phi_{s}\left(\tilde{r}_{0},0,\xi_{0}\right)$ intersects $\varSigma_{+}$  
again and $\tilde{r}_{1}$ is such that 
$\phi_{s}\left(\tilde{r}_{0},0,\xi_{0}\right)=\left(\tilde{r}_{1},0,\xi_{1}\right)$.
We set $\xi_1=\infty$ for points $\left(\tilde{r}_0,\xi_0\right)$ which do not 
intersect $\varSigma_{+}$ anymore in the future and define 
$D_0\subset\varSigma_{+}$  as the set of points for which $\xi_1<\infty$. In the 
unperturbed problem ($G\to\infty$) one easily deduces, using the conservation of 
energy, that $\varSigma_+$ is divided in two open sets, corresponding to 
initial conditions leading to hyperbolic and elliptic motions, 
whose common boundary is  the curve in which the 
homoclinic manifold \eqref{eq:unperturbed_homoclinic_first} intersects 
$\varSigma_{+}$.  In this case, $D_0$ corresponds to the set of initial 
conditions leading to elliptic motions.

In order to characterize the set $D_0$ in the full problem \eqref{eq:rescaled_hamiltonian} we make use of the following proposition, already proven in \cite{Brandao2017}, which describes the intersection  $\mathcal{W}^{s,u}\cap \varSigma_{+}$. 
\begin{prop}\label{prop:proposiciónVidalintersecc}
The stable  manifold $\mathcal{W}^s$ intersects $\varSigma_{+}$ backwards for the first time in a simple curve 
\begin{equation}
\tilde{\gamma}^s=\left\{\left(\tilde{r}_{0}^s\left(\xi_0\right),\xi_0\right)\in\varSigma_{+}\colon
 \tilde{r}_0^s\left(\xi_0+2\pi\right)=\tilde{r}_0^s\left(\xi_0\right)\right\}.
\end{equation}
 Analogously, the unstable manifold  $\mathcal{W}^u$ intersects $\varSigma_{+}$ forward for the first time in a simple curve   \begin{equation}
 \tilde{\gamma}^u=\left\{\left(\tilde{r}_{0}^u\left(\xi_0\right),\xi_0\right)\in\varSigma_{+}\colon \tilde{r}_0^{u}\left(\xi_0+2\pi\right)=\tilde{r}_0^{u}\left(\xi_0\right)\right\}.\end{equation}
\end{prop}

\begin{rem}
From Theorem \ref{thm:Main Theorem 2-1} we deduce that the curves $\tilde{\gamma}^{s,u}$ described in Proposition \ref{prop:proposiciónVidalintersecc} intersect transversally, a fact which is crucial for the proof of Theorem \ref{thm:Main Theorem 2-2}.
\end{rem}

The curve $\tilde{\gamma}^{s}$ divides $\varSigma_{+}$ in two connected components. One of these components correspond to $D_0$  and the other component  consists of initial conditions leading to orbits which do not intersect $\varSigma_{+}$ again and which escape to infinity with positive asymptotic radial velocity. We also define the set  $D_1\subset\varSigma_{+}$ of initial conditions $\left(\tilde{r}_0,\xi_0\right)$, in which the map $\psi^{-1}$ is well defined. A similar argument to the one above using $\tilde{\gamma}^{u}$ instead of $\tilde{\gamma}^{s}$ can be used to identify this set.

Once we have identified $D_0$ and $D_1$, given a point $\left(\tilde{r}_{0},\xi_{0}\right)\in D_{0}\cap D_1$
we consider the sequence of consecutive times $\xi_{n}$ given by 
$\psi^{n}\left(\tilde{r}_{0},\xi_{0}\right)=(\tilde{r}_{n},\xi_{n})$
for $n\in\mathbb{Z}$ (whenever they exist) to define the sequence of 
integers 
\[
a_{n}=\left[\frac{\xi_{n}-\xi_{n-1}}{2\pi}\right],
\]
where $\left[\cdot\right]$ defines the integer part. Thus, $a_{n}\in\mathbb{N}$
measures the number of binary collisions of the primaries between
consecutive approaches of the third body. We introduce some technical
concepts needed for stating the theorem that establishes the existence
of symbolic dynamics on a subset of the closure $D_{0}\cap D_{1}$
by conjugating $\psi$ with the shift acting on a space of doubly
infinite sequences.

Let $A$ denote the set of all doubly infinite sequences 
\[
a=\left(\dots a_{-2},a_{-1},a_{0};a_{1},a_{2}\dots\right)
\]
 of elements $a_{n}\in\mathbb{N}$. Equipping $A$ with the product
topology, the shift $\sigma:A\rightarrow A$ given by 
\begin{equation}
\sigma\left(\left\{ a_{n}\right\} _{n\in\mathbb{Z}}\right)=\left\{ a_{n-1}\right\} _{n\in\mathbb{Z}}\label{eq:full shift}
\end{equation}
is a homeomorphism.

We can define the compactification $\bar{A}$ of $A$ admitting elements
of the following type: For $\alpha,\beta$ integers satysfying $\alpha\leq0,$
$\beta\geq1,$ let 
\[
a=\left(\infty,a_{\alpha+1},\dots,a_{\beta-1},\infty\right)\quad a_{n}\in\mathbb{N}.
\]
 We also admit half infinite sequences with $\alpha=-\infty,$ $\beta<\infty$
or $\alpha>-\infty,$ $\beta=\infty.$ It is possible to extend the
topology defined on $A$ to $\bar{A}$ in a way such that the shift
\eqref{eq:full shift} is a homeomorphism when restricted to 
\[
\bar{A}_{0}=\left\{ a\in\bar{A}\colon a_{0}\neq\infty\right\} 
\]
 (see \cite{moser1973stable} for details). 

The proof of the following theorem, from which Theorems \ref{thm:Main_theorem}
and \ref{thm:Main_Theorem_b} are deduced, follows from direct adaptation
of the ideas presented in \cite{moser1973stable} for the Sitnikov
problem. The main ingredients are the transversal intersection of
the curves $\gamma^{s,u}$ and a $C^{1}$ Lambda-Lemma for
the parabolic invariant manifold $\varLambda$. This Lambda-Lemma
follows from a careful analysis of the dynamics near $\varLambda$
using McGehee coordinates which map neighbourhoods of infinity into
bounded neighborhoods of the origin.
\begin{thm}
\label{thm:Main Theorem 2-2}There exists a set $S\subset\left(D_{0}\cap D_{1}\right)$
which is invariant under the Poincar{\'e} map $\psi$ defined in \eqref{eq:map for the horseshoe}
and such that its restriction to $S$, is conjugated to the shift
$\sigma$ defined in \eqref{eq:full shift}. That is, there exists
an homeomorphism $\chi:A\to S$ such that 
\[
\psi\chi=\chi\sigma.
\]

Moreover, $\chi$ can be extended to $\bar{\chi}:\bar{A}\to\bar{S}$
such that 
\[
\psi\bar{\chi}=\bar{\chi}\sigma
\]
 if both sides are restricted to $\bar{A}_{0}.$
\end{thm}

In other words, to each point 
$p=\left(r_{0},\xi_{0}\right)\in S$ we
associate a sequence $a(p)\in A$ which codifies the time between
successive intersections of the flow $\phi_{s}\left(r_{0},0, \xi_{0}\right)$
with $\varSigma_+.$ In this setting, the connection between Theorem
\ref{thm:Main_theorem} and Theorem \ref{thm:Main_Theorem_b} becomes
clear. The first part of Theorem \ref{thm:Main Theorem 2-2} corresponds
to sequences
\begin{itemize}
\item $a\left(p\right)=\left(\dots a_{-2},a_{-1},a_{0},a_{1},\dots\right)$
with $a_{n}\in\mathbb{\mathbb{N}}$ for all $n\in\mathbb{Z}.$ These
represent orbits which perform an infinite number of ``close'' approaches
to the line where the primaries move both in the past and in the future.
From this result we deduce the existence of any past-future combination
of bounded ($\sup_{n\in{\mathbb{Z}}}a_n<\infty$) and oscillatory ($\limsup_{n\to\pm\infty}a_n=\infty$) motions.
\end{itemize}
The second part of the theorem, concerns sequences of the following
type 
\begin{itemize}
\item $a\left(p\right)=\left(\infty,a_{-k},a_{-k+1},\dots\right)$ with
$a_{n}\in\mathbb{\mathbb{N}}$ for all $n>-k$, which represent capture
orbits, i.e., orbits where the third body comes from infinity at $t\rightarrow-\infty$
and remains revolving around the line of primaries for all future
times. In particular, we obtain orbits which are hyperbolic or parabolic
in the past and bounded or oscillatory in the future.
\item $a\left(p\right)=\left(\dots a_{l-1},a_{l},\infty\right)$ with $a_{n}\in\mathbb{\mathbb{N}}$
for all $n<l$. In this case the third body performed an infinite
number of oscillations in the past but escapes to infinity as $t\rightarrow\infty$.
These sequences correspond to orbtis which are bounded or oscillatory
in the past and parabolic or hyperbolic in the future.
\item $a\left(p\right)=\left(\infty,a_{k},\dots,a_{l},\infty\right)$ with
$a_{n}\in\mathbb{\mathbb{N}}$ for all $n\in\mathbb{Z}$ which corresponds
to orbits coming from infinity, revolving around the primaries a finite number of times
and escaping again to infinity as $t\rightarrow\infty$. They correspond
to past-future combinations of parabolic and hyperbolic motions.
\end{itemize}
Finally, we point out that the existence of infinitely many periodic orbits in the REI3BP is deduced from Theorem   \ref{thm:Main_theorem} since fixed points for the shift correspond to periodic orbits of the Hamiltonian \eqref{eq:rescaled_hamiltonian}.

\section{The invariant Manifolds as graphs}

\subsection{\label{sec:The-integrable-system}The unperturbed homoclinic solution}

For the unperturbed problem, $G\to\infty$ in \eqref{eq:rescaled_hamiltonian}, the equations
of motion reduce to
\begin{equation}\label{eq:unperturbed eqs}
\begin{split}
\frac{\text{d}\tilde{r}}{\text{d}v} & =\tilde{y}\\
\frac{\text{d}\tilde{y}}{\text{d}v} & =\frac{1}{\tilde{r}^{3}}-\frac{1}{\tilde{r}^{2}}.
\end{split}
\end{equation}
In this case the infinity manifolds $\mathcal{W}_{\infty}^{s,u}$
associated to $\varLambda$ coincide along the two dimensional homocinic manifold $\tilde{z}_{h}$  introduced in \eqref{eq:unperturbed_homoclinic_first}. The (complex) singularities of $\tilde{z}_h\left(v,\xi\right)$ will be crucial for studying the existence
of the invariant manifolds of the perturbed problem in certain complex
domains. Thus, we state the following results, which were already obtained in \cite{Martinez1994}.
 \begin{enumerate}
   \item \label{itm:Homoclinic Infinity}  The homoclinic solution \eqref{eq:unperturbed_homoclinic_first}
behaves as 
\[
\tilde{r}_{h}\left(v\right)\sim3v^{2/3},\qquad\tilde{y}_{h}\left(v\right)\sim2v^
{-1/3} \qquad\text{ as }\qquad\left|v\right|\rightarrow\infty. 
\]
\item The homoclinic solution \eqref{eq:unperturbed_homoclinic_first} is a real 
analytic function of $v$ with singularities at $v=\pm i/3$. 
   \item \label{itm:Homoclinic Singul} Close to its singularities, the homoclinic solution \eqref{eq:unperturbed_homoclinic_first} behaves as
\[
\tilde{r}_{h}\left(v\right)\sim 
C\left(v\mp\frac{i}{3}\right)^{1/2},\qquad\tilde{y}_{h}\left(v\right)\sim\frac{C
}{2}\left(v\mp\frac{i}{3}\right)^{1/2},\quad\text{ where }\quad C^{2}=\pm2i.
\]
\end{enumerate}

\subsection{\label{sec:Parametrizations-of-the} The perturbed invariant
manifolds and their difference }
In this section we look for parametrizations of the infinity manifolds
$\mathcal{W}_{\infty}^{u,s}$ in certain complex domains defined below.
More concretely we look for graph parametrizations of $\mathcal{W}_{\infty}^{u,s}$
as solutions to a PDE. To do this we observe that the canonical form
$\lambda=\tilde{r}\text{d}\tilde{y}-\tilde{H}\text{d}s$ is closed
on the infinity manifolds (since the infinity manifolds are invariant
by the flow it is enough to check that $\text{d}\lambda$ is null
on $\varLambda$). Then, one can see $\lambda$ as the differential
of a function $S\left(\tilde{r},\xi\right)$ such that 
\[
\partial_{\tilde{r}}S=\tilde{y}\qquad G^{3}\partial_{\xi}S=-\tilde{H}
\]
or, putting this together, as a solution of the Hamilton-Jacobi equation
\[
G^{3}\partial_{\xi}S+H\left(\tilde{r},\partial_{\tilde{r}}S,\xi\right)=0.
\]
We write $S=S_{0}+S_{1}$ where $S_{0}$ is the solution to the unperturbed
problem 
\[
G^{3}\partial_{\xi}S_{0}+\frac{\left(\partial_{\tilde{r}}S_{0}\right)^{2}}{2}+\frac{1}{2\tilde{r}^{2}}-\frac{1}{\tilde{r}}=0
\]
and perform the change of variables 
\begin{equation}
\left(\tilde{r},\xi\right)\mapsto\left(\tilde{r}_{h}\left(v\right),\xi\right).\label{eq:change v,xi}
\end{equation}
Then, the equation for $T_{1}\left(v,\xi\right)=S_{1}\left(\tilde{r}_{h}\left(v\right),\xi\right)$
becomes 
\begin{equation}
\partial_{v}T_{1}+\frac{\text{1}}{2\tilde{y}_{h}{}^{2}}\left(\partial_{v}T_{1}\right)^{2}+G^{3}\partial_{\xi}T_{1}+V\left(v,\xi\right)=0,\label{eq:eq T1}
\end{equation}
where 
\begin{equation}\label{eq:perturb potential}
V\left(v,\xi\right)=U\left(\tilde{r}_h(v),\xi\right).
\end{equation}
 Note that the change of variables \eqref{eq:change v,xi} implies
that we are looking for parametrizations of the stable and unstable
manifolds of the form 
\begin{equation}
\begin{split}\tilde{r} & =\tilde{r}_{h}\left(v\right)\\
\tilde{y} & =\tilde{y}_{h}\left(v\right)+\frac{1}{\tilde{y}_{h}\left(v\right)}\partial_{v}T_{1}^{u,s}
\end{split}
\label{eq:H-J parametrization}
\end{equation}
where $\tilde{r}_{h}\left(v\right),\tilde{y}_{h}\left(v\right)$ correspond
to the unperturbed homoclinic \eqref{eq:unperturbed_homoclinic_first}
and $T_{1}^{u,s}\left(v,\xi\right)$ are solutions of equation \eqref{eq:eq T1}
with asymptotic boundary condition for the unstable manifold 
\begin{equation}
\lim_{v\rightarrow-\infty}\frac{1}{\tilde{y}_{h}\left(v\right)}\partial_{v}T_{1}^{u}=0\label{eq:Boundary condition}
\end{equation}
and the analogous one for the stable manifold. Once we show the existence
of the unstable manifold, the existence of the stable one is guaranteed
by symmetry. Indeed, if $T_{1}\left(v,\xi\right)$ is a solution of
\eqref{eq:eq T1}, $-T_{1}\left(-v,-\xi\right)$ is also a solution
satisfying the opposite boundary condition. 

Before going into the analysis of the existence of the generating
functions $T_{1}^{u,s}$ we recall that our goal is to have a first
asymptotic approximation of the distance between the infinity manifolds
which now boils down to obtain an asymptotic formula for $\partial_{v}\left(T_{1}^{u}-T_{1}^{s}\right)$.
To this end, we introduce the Melnikov potential 
\begin{equation}
L\left(v,\xi;G\right)=\int_{-\infty}^{\infty}V\left(\tilde{r}_{h}\left(v+s\right),\xi+G^{3}s\right)\text{d}s,\label{eq:Melnikov potential}
\end{equation}
which, as we state in Theorem \ref{thm:Theorem Melnikov-generating functions} 
below approximates to first order the difference 
$\varDelta=T_{1}^{s}-T_{1}^{u}$. 

We point out that the parametrization \eqref{eq:H-J parametrization}
becomes undefined at $v=0$ since we have fixed $v$ such that $\tilde{y}_{h}\left(0\right)=0.$
Since in order to measure $\partial_{v}\left(T_{1}^{u}-T_{1}^{s}\right)$
we need both functions to be defined in a common domain we will introduce
a different parametrization to extend the unstable manifold across
$v=0.$ This is discussed in full detail in Section \ref{sec:Existence-of-the}.

The next proposition gives the first asymptotic term of the Melnikov
potential and will be proved in Section \ref{sec:Computation-of-the}.
\begin{prop}
\label{prop:Melnikov potential}The function function $L\left(v,\xi;G\right)$
defined in \eqref{eq:Melnikov potential} satisfies 
\[
L\left(v,\xi;G\right)=L^{[0]}\left(G\right)+2\sum_{l=1}^{\infty}L^{[l]}\left(G\right)\cos\left(l\left(\xi-G^{3}v\right)\right),
\]
where 
\begin{align*}
L^{[1]}\left(G\right) & =-J_{1}\left(1\right)\sqrt{2\pi}G^{\nicefrac{-5}{2}}e^{\frac{-G^{3}}{3}}\left(1+\mathcal{O}\left(G^{-\nicefrac{3}{2}}\right)\right)\\
\left|L^{[l]}\left(G\right)\right| & \leq KG^{\nicefrac{-5}{2}}e^{l-\nicefrac{1}{2}}e^{\frac{-\left|l\right|G^{3}}{3}},\quad\text{for }l>1,
\end{align*}
with $J_{1}$ the first Bessel function of the first kind and $K>0$
a constant independent of $G.$ 
\end{prop}

\begin{thm}
\label{thm:Theorem Melnikov-generating functions} Choose any $0<v_{-}<v_{+}<\infty$. Then, there exists $K>0$  such that for any $v\in\left[v_{-},v_{+}\right]$ and for any $G$ large enough, the generating functions $T_{1}^{u,s}\left(v,\xi\right)$ satisfy

\[
\left|T_{1}^{s}\left(v,\xi\right)-T_{1}^{u}\left(v,\xi\right)-L\left(v,\xi\right)-E\right|\leq KG^{-\nicefrac{7}{2}}e^{\frac{-G^{3}}{3}},
\]
where $E\in\mathbb{R}$ is a constant and 
\[
\left|\partial_{v}\left(T_{1}^{s}\left(v,\xi\right)-T_{1}^{u}\left(v,\xi\right)\right)-\partial_{v}L\left(v,\xi\right)\right|\leq KG^{-\nicefrac{1}{2}}e^{\frac{-G^{3}}{3}}.
\]
\end{thm}
From Proposition \ref{prop:Melnikov potential}, Theorem \ref{thm:Theorem Melnikov-generating functions}
and Equation \eqref{eq:H-J parametrization} we deduce Theorem \ref{thm:Main Theorem 2-1}.
We devote Sections \ref{sec:Existence-of-the} and \ref{sec:The-difference-between}
to the proof of Theorem \ref{thm:Theorem Melnikov-generating functions}. 

\section{\label{sec:Existence-of-the} The invariant manifolds in complex domains}

The classical procedure when studying exponentially small splitting
of separatrices is to look for the functions $T_{1}^{u}$ and $T_{1}^{s}$
in a complex common domain $D\times\mathbb{T}_{\sigma}$ where $D\subset\mathbb{C}$
is a connected domain which reaches a neighborhood of size $\mathcal{O}\left(G^{-3}\right)$
(recall that the period of the perturbation \eqref{eq:potential_main_order_term} is $2\pi/G^3$) of the singularities of the unperturbed separatrix, i.e., $v=\pm i/3$ (see Section \ref{sec:The-integrable-system})
and 
\[
\mathbb{T}_{\sigma}=\left\{ \xi\in\mathbb{C}/2\pi\mathbb{Z}:\left|\text{Im}(\xi)\right|<\sigma\right\} .
\]
The idea behind this approach is that for $v\in\mathbb{R}$ we will
get exponentially small bounds on the distance $d\left(v,\xi\right)$
between the invariant manifolds if we show that $d$ is a quasiperiodic function in some suitable coordinates  and
we manage to bound $\left|d\right|$  in a connected domain $D$ which contains a subset of the real line and gets close to the singularities $v=\pm i/3$ .
\begin{figure}
\begin{center}
\includegraphics[scale=0.4]{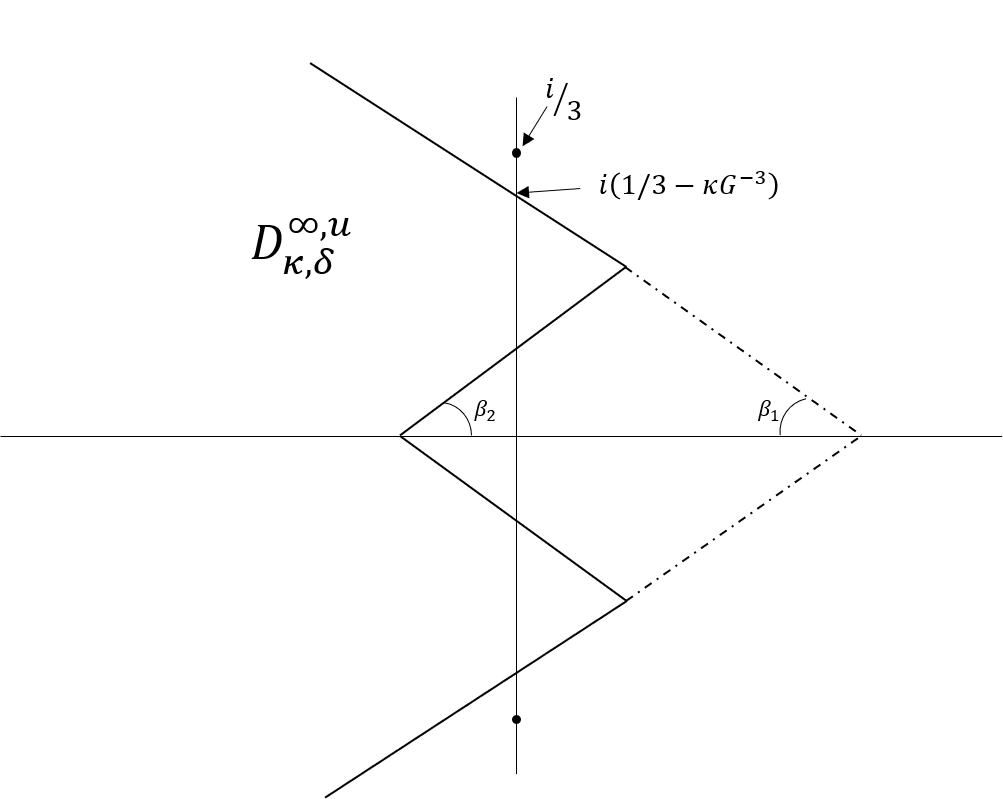}\caption{\label{fig:Domain infinity}The domain $D_{\kappa,\delta}^{\infty,u}$
defined in (\ref{eq:infinity domain}). }
\end{center}
\end{figure}

Since boundary conditions are imposed at infinity, we need to solve
the equation \eqref{eq:eq T1} for $T_{1}^{u}$ (resp. $T_{1}^{s}$)
in a complex unbounded domain reaching $v\to-\infty$ (resp. $v\to\infty$).
On the other hand, in order to measure their difference we need them
to be defined in a common domain, we need to extend one of them across $v=0.$
However, the equation \eqref{eq:eq T1} becomes singular at $v=0$
since $\tilde{y}_{h}\left(0\right)=0.$ To overcome this problem we
divide the process of extension of the invariant manifolds into three
steps.

We first solve equation \eqref{eq:eq T1} together with the
boundary condition $\eqref{eq:Boundary condition}$ in the domain

\begin{equation}
D_{\kappa,\delta}^{\infty,u}=\left\{ v\in\mathbb{C}\colon\left|\text{Im}(v)\right|<-\text{tan}\beta_{1}\text{Re}(v)+1/3-\kappa G^{-3},\right.\left.\left|\text{Im}(v)\right|>\text{tan}\beta_{2}\text{Re}(v)+1/6-\delta\right\} ,\label{eq:infinity domain}
\end{equation}
which does not contain $v=0$ and where $\kappa,\delta$ and $\beta_{1},\beta_{2}\in\left(0,\pi/2\right)$
are fixed independently of $G$ (see Figure \ref{fig:Domain infinity}).
One can check that for $\delta\in\left(0,1/12\right)$, $\kappa\sim\mathcal{O}\left(1\right),$
we can always find $G$ big enough such that this domain is non empty.
Once the existence of $T_{1}^{u}$ in the domain $D_{\kappa,\delta}^{\infty,u}$
is proven, we exploit the symmetry of equation \eqref{eq:eq T1} under
the map $\left(v,\xi\right)\to\left(-v,-\xi\right)$ to atutomatically
deduce the existence of $T_{1}^{s}$ in the domain 
\begin{equation}
D_{\kappa,\delta}^{\infty,s}=\left\{ v\in\mathbb{C}\colon\left|\text{Im}(v)\right|<\text{tan}\beta_{1}\text{Re}(v)+1/3-\kappa G^{-3},\right.\left.\left|\text{Im}(v)\right|>-\text{tan}\beta_{2}\text{Re}(v)+1/6-\delta\right\} .\label{eq:infinity domain stable}
\end{equation}

The next step is to perform the analytical continuation of $T_{1}^{u}$
accross the imaginary axis. Thus, we would have both invariant manifolds
defined on a common domain (this domain will be contained in $D_{\kappa,\delta}^{\infty,s}$
where $T_{1}^{s}$ is already defined). Since $y_{h}\left(0\right)=0$,
the equation \eqref{eq:eq T1} becomes singular at $v=0$ so we change
to a parametrization invariant by the flow in the bounded domain 
\begin{equation}
D_{\rho,\kappa,\delta}=D_{\kappa,\delta}^{\infty,u}\cap\left(\mathrm{Re}(v)>-\rho\right)\label{eq:Domain D rho kappa delta}
\end{equation}
for some finite $\rho>0.$ Then, we use the flow $\phi_{s}$ associated
to the system \eqref{eq:eqs of motion} to extend the unstable manifold
$T_{1}^{u}$ to the domain 

\begin{figure}
\begin{center}
\includegraphics[scale=0.4]{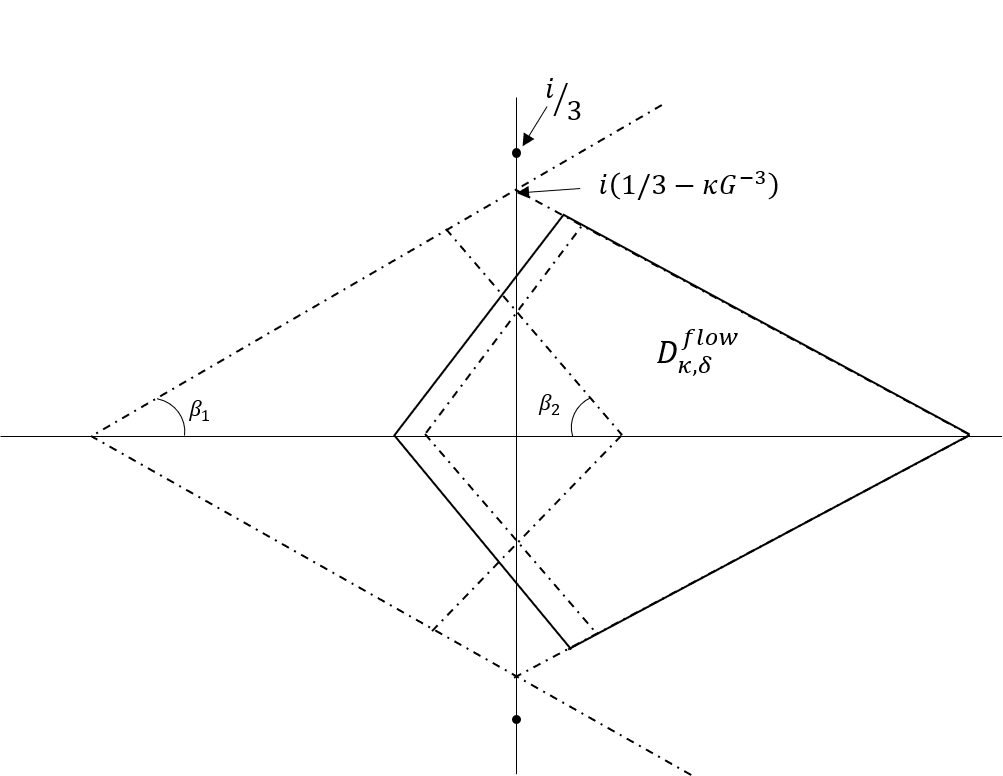}\caption{\label{fig:Domain flow}The domain $D_{\kappa,\delta}^{\mathrm{flow}}$
defined in (\ref{eq:flow domain}). }
\end{center}
\end{figure}

\begin{equation}
D_{\kappa,\delta}^{\mathrm{flow}}=\left\{ v\in\mathbb{C}\colon\left|\text{Im}(v)\right|<-\tan\beta_{1}\text{Re}(v)+1/3-\kappa G^{-3},\right.\left.\left|\text{Im}v\right|<\tan\beta_{2}\text{Re}(v)+1/6+\delta\right\} \label{eq:flow domain}
\end{equation}
which contains $v=0$ (see Figure \ref{fig:Domain flow}). Then we
go back to the original parametrization in a ``boomerang domain'' 

\begin{equation}\label{eq:boomerang domain}
\begin{split}
D_{\kappa,\delta} = & \left\{ v\in\mathbb{C}:\left|\text{Im}(v)\right|<-\text{tan}\beta_{1}\text{Re}(v)+1/3-\kappa G^{-3}, \left|\text{Im}(v)\right|<\text{tan}\beta_{1}\text{Re}(v)+1/3-\kappa G^{-3}\right.,\\
& \left.\left|\text{Im}(v)\right|>-\text{tan}\beta_{2}\text{Re}(v)+1/6-\delta\right\},
\end{split}
\end{equation}
(which does not contain $v=0$) in order to measure the distance between the stable and unstable manifold.

\begin{figure}
\begin{center}
\includegraphics[scale=0.4]{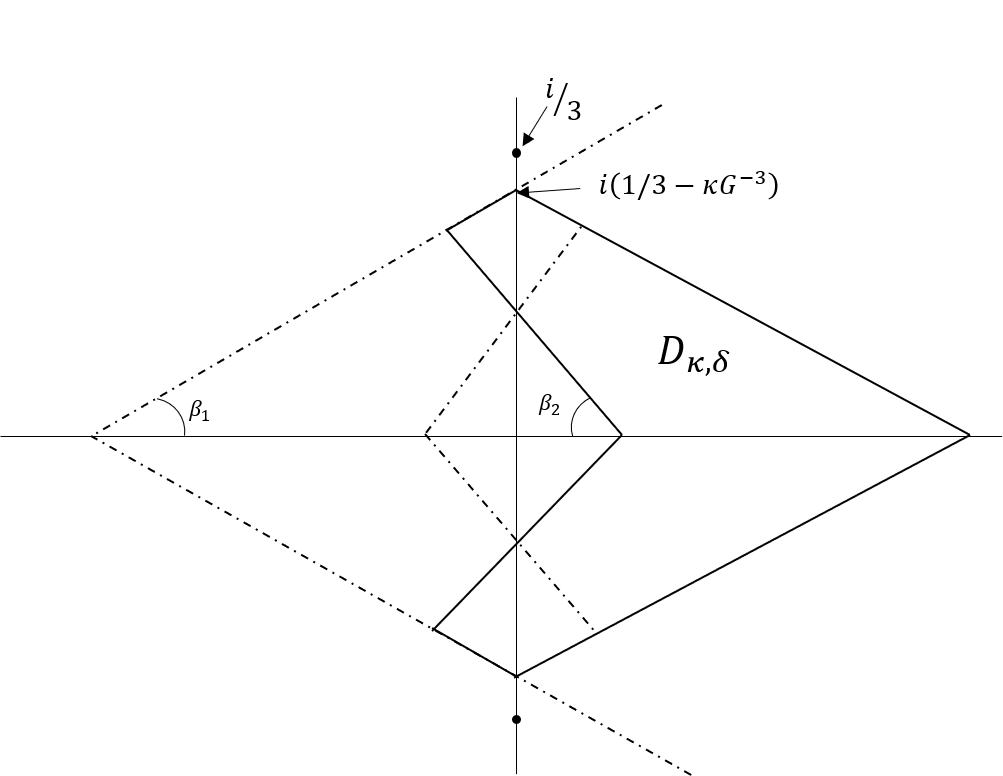}\caption{\label{fig:Domain Boomerang}The domain $D_{\kappa,\delta}$ defined
in (\ref{eq:boomerang domain})}.
\end{center}
\end{figure}

\subsection{Existence of the invariant manifolds close to infinity}

In order to prove existence of the invariant manifolds we rewrite
equation \eqref{eq:eq T1} as a fixed point equation in a suitable
Banach space. We start by defining the linear operator 
\begin{equation}
\mathcal{L}=\partial_{v}+G^{3}\partial_{\xi}\label{eq:linear operator L}
\end{equation}
so equation \eqref{eq:eq T1} reads 
\begin{equation}
\mathcal{L}\left(T_{1}^{u,s}\right)=\mathcal{F}\left(T_{1}^{u,s}\right)\qquad\mathcal{\text{where\qquad}F}\left(T_{1}^{u,s}\right)=-\frac{1}{2\tilde{y}_{h}^{2}}\left(\partial_{v}T_{1}^{u,s}\right)^{2}-V\left(v,\xi\right).\label{eq:eq infinity}
\end{equation}
We introduce the left inverse operators 
\begin{equation}
\begin{split}\mathcal{G}^{u}\left(f\right)\left(v,\xi\right)= & \int_{-\infty}^{0}f\left(v+s,\xi+G^{3}s\right)\text{d}s\\
\mathcal{G}^{s}\left(f\right)\left(v,\xi\right)= & \int_{+\infty}^{0}f\left(v+s,\xi+G^{3}s\right)\text{d}s,
\end{split}
\label{eq:inverse operator}
\end{equation}
 so we can rewrite equation \eqref{eq:eq infinity} as the fixed point
equation 
\begin{equation}
T_{1}^{u,s}=\mathcal{G}^{u,s}\circ\mathcal{F}\left(T_{1}^{u,s}\right).\label{eq: Fixed point equation}
\end{equation}
\begin{rem*}
Throughout this section we will only work with the unstable manifold
so we will omit the superindex $u$ and write $D_{\kappa,\delta}^{\infty},\ T_{1}$
and $\mathcal{G}$ instead of $D_{\kappa,\delta}^{\infty,u},\ T_{1}^{u}$
and $\mathcal{G}^{u}$ if there is no possible confusion.
\end{rem*}
We look for solutions of this equation in the Banach spaces 
\begin{equation}
\mathcal{Z}_{\nu,\mu}^{\infty}=\left\{ h:D_{\kappa,\delta}^{\infty}\times\mathbb{T}_{\sigma}\rightarrow\mathbb{C}\colon h\;\text{is real analytic},\:\left\Vert h\right\Vert _{\nu,\mu}<\infty\right\} ,\label{eq:Banach spaces infinity}
\end{equation}
where 
\[
\left\Vert h\right\Vert _{\nu,\mu}=\sum_{l\in\mathbb{Z}}\left\Vert h^{[l]}\right\Vert _{\nu,\mu}e^{\left|l\right|\sigma}
\]
and 
\[
\left\Vert h^{[l]}\right\Vert _{\nu,\mu}=\sup_{v\in D_{\kappa,\delta}^{\infty}\setminus D_{\rho,\kappa,\delta}}\left|v^{\nu}h^{[l]}\left(v\right)\right|+\sup_{v\in D_{\rho,\kappa,\delta}}\left|\left(v^{2}+1/9\right)^{\mu}h^{[l]}\left(v\right)\right|.
\]
Notice that the first term takes account of the behaviour at infinity
and the second one of the behaviour near the singularities since $v^{2}+1/9=\left(v-i/3\right)\left(v+i/3\right)$.
As we see from \eqref{eq: Fixed point equation} we will also need
to take control on the derivatives so we introduce 
\begin{equation}
\mathcal{\tilde{Z}}_{\nu,\mu}^{\infty}=\left\{ h:D_{\kappa,\delta}^{\infty}\times\mathbb{T}_{\sigma}\rightarrow\mathbb{C}\colon h\;\text{is real analytic},\:\left\llbracket h\right\rrbracket _{\nu,\mu}<\infty\right\} ,\label{eq: banach space derivative}
\end{equation}
where 
\[
\left\llbracket h\right\rrbracket _{\nu,\mu}=\left\Vert h\right\Vert _{\nu,\mu}+\left\Vert \partial_{v}h\right\Vert _{\nu+1,\mu+1}.
\]

The following lemma provides estimates for the norm of the perturbative
potential.
\begin{lem}
\label{lem: Potential estimates}Let V be the perturbative potential
defined in \eqref{eq:perturb potential}. Then, for $G$ large
enough we have that 
\[
\left\Vert V\right\Vert _{2,3/2}\leq KG^{-4}
\]
for a constant $K>0$ independent of $G.$
\end{lem}

\begin{proof}
Since the domain $D_{\kappa,\delta}^{\infty}$ reaches a neighbourhood
of order $\mathcal{O}\left(G^{-3}\right)$ of $v=\pm i/3$ we have
that for $G$ sufficiently large 
\[
\left|\frac{1}{G^{4}\tilde{r}_{h}^{2}\left(v\right)}\right|\leq KG^{-1},
\]
 for $K>0$ independent of $G.$ Therefore, from \eqref{eq:potential_main_order_term}
we deduce that for all $\left(v,\xi\right)\in D_{\kappa,\delta}^{\infty}\times\mathbb{T}_{\sigma}$
\[
\left|V\left(v,\xi\right)\right|\leq\frac{K}{G^{4}\left|\tilde{r}_{h}\left(v\right)\right|^{3}}.
\]
 The conclusion follows now using the asymptotic expressions for $\tilde{r}_{h}\left(v\right)$
obtained in Section \ref{sec:The-integrable-system}.
\end{proof}
We also state algebra-like properties for these spaces, which are
straightforward from their definition and will be useful when dealing
with the fixed point equation. 
\begin{lem}
\label{lem:Algebra-properties-for}Let $\mathcal{Z}_{\nu,\mu}^{\infty}$
be the Banach spaces defined in \eqref{eq:Banach spaces infinity}.
Then

i) If $h\in\mathcal{Z}_{\nu,\mu}^{\infty}$ and $g\in\mathcal{Z}_{\nu',\mu'}^{\infty}$
then $hg\in\mathcal{Z}_{\nu+\nu',\mu+\mu'}^{\infty}$ with 
\[
\left\Vert hg\right\Vert _{\nu+\nu',\mu+\mu'}\leq\left\Vert h\right\Vert _{\nu,\mu}\left\Vert g\right\Vert _{\nu',\mu'}.
\]

ii) If $h\in\mathcal{Z}_{\nu,\mu}^{\infty}$, then $h\in\mathcal{Z}_{\nu-\alpha}^{\infty}$
for $\alpha>0$ with 
\[
\left\Vert h\right\Vert _{\nu-\alpha,\mu}\leq K\left\Vert h\right\Vert _{\nu,\mu}.
\]

iii) If $h\in\mathcal{Z}_{\nu,\mu}^{\infty}$ then, for $\alpha>0$
we have that $h\in\mathcal{Z}_{\nu,\mu-\alpha}^{\infty}$ with 
\[
\left\Vert h\right\Vert _{\nu,\mu-\alpha}\leq KG^{3\alpha}\left\Vert h\right\Vert _{\nu,\mu}.
\]

iv) If $h\in\mathcal{Z}_{\nu,\mu}^{\infty}$ then, for $\alpha>0$
we have that $h\in\mathcal{Z}_{\nu,\mu+\alpha}^{\infty}$ with 
\[
\left\Vert h\right\Vert _{\nu,\mu+\alpha}\leq K\left\Vert h\right\Vert _{\nu,\mu}.
\]
\end{lem}

The following lemma provide estimates for the inverse operator. The
proof follows the exact same lines as in Lemma 5.5. in \cite{Guardia2010}
(see also \cite{Baldoma2012}).
\begin{lem}
\label{lem:Inverse operator lemma}The operator $\mathcal{G}$ defined
on \eqref{eq:inverse operator} satisfies the following properties

i) For any $\nu>1$, $\mu\>1$, $\mathcal{G}:\mathcal{Z}_{\nu,\mu}\rightarrow\mathcal{Z}_{\nu-1,\mu-1}$
is well defined, linear and satisfies $\mathcal{L}\circ\mathcal{G}=\mathrm{Id}.$

 ii)If $h\in\mathcal{Z}_{\nu,\mu}$ for some $\nu>1$, $\mu>1,$ then
\begin{equation}
\left\Vert \mathcal{G}\left(h\right)\right\Vert _{\nu-1,\mu-1}\leq K\left\Vert h\right\Vert _{\nu,\mu}.\label{eq:inverse bound2}
\end{equation}

iii) If $h\in\mathcal{Z}_{\nu,\mu}$ for some $\nu\geq1$, $\mu\geq1,$
then 
\begin{equation}
\left\Vert \partial_{v}\mathcal{G}\left(h\right)\right\Vert _{\nu,\mu}\leq K\left\Vert h\right\Vert _{\nu,\mu}.\label{eq:inverse derivative bound}
\end{equation}
\end{lem}

Now we are ready to solve the fixed point equation.
\begin{thm}
\label{thm:Existence infinity manifold close to infinity}Fix $\kappa>0,\ \delta>0$
and $\sigma>0.$ Then, for $G$ large enough the fixed point equation
\eqref{eq: Fixed point equation} has a unique solution $T_{1}^{u}$
on $D_{\kappa,\delta}^{\infty}\times\mathbb{T}_{\sigma}$ which satisfies
\[
\left\llbracket T_{1}^{u}\right\rrbracket _{1,1/2}\leq b_{0}G^{-4}
\]
with $b_{0}>0$ independent of $G.$ Moreover, if we define the function
\[
L_{1}^{u}\left(v,\xi\right)=\mathcal{G}^{u}\left(V\right)\left(v,\xi\right)
\]
we have 
\begin{equation}
\left\Vert T_{1}^{u}-L_{1}^{u}\right\Vert _{1,1/2}\leq KG^{-13/2}\label{eq:difference with potential}
\end{equation}
where $K>0$ is independent of $G.$
\end{thm}

\begin{proof}
We show that $T_1$  is the unique solution the fixed point equation \eqref{eq: Fixed point equation}. For that we first check that the operator $\mathcal{G}\circ\mathcal{F}$ is well defined from $\mathcal{\tilde{Z}}_{1,1/2}$
to itself. Indeed, from Lemma \ref{lem: Potential estimates} we have
that
\[
\left\Vert V\right\Vert _{2,3/2}\leq KG^{-4}.
\]
Then, the result follows from direct application of the properties of the homoclinic solution stated in Section \ref{sec:The-integrable-system}, the algebra properties
of the norm stated in Lemma \ref{lem:Algebra-properties-for} and
Lemma \ref{lem:Inverse operator lemma} since we obtain that for $h\in \mathcal{\tilde{Z}}_{1,1/2}$
\begin{equation}
\left\llbracket \mathcal{G}\circ\mathcal{F}\left(h\right)\right\rrbracket _{1,1/2}\leq K \min\left(\left\llbracket h\right\rrbracket_{1,1/2}, G^{-4}\right)
\end{equation} 
for some $K>0$ independent of $G$. In particular we deduce that there exists $b_{0}>0$ independent of $G$ such that 
\[
\left\llbracket \mathcal{G}\circ\mathcal{F}\left(0\right)\right\rrbracket _{1,1/2}\leq\frac{b_{0}}{2}G^{-4}.
\]
Then in order to show existence and uniqueness of solutions it is
enough to show that the map $\mathcal{G}\circ\mathcal{F}$ is contractive
on the ball  $B\left(b_{0}G^{-4}\right)\subset\mathcal{\tilde{Z}}_{1,1/2}$ centered at 0. For that purpose we write 
\[
\mathcal{F}\left(h_{2}\right)-\mathcal{F}\left(h_{1}\right)=\frac{1}{2y_{h}^{2}}\left(\partial_{v}h_{1}+\partial_{v}h_{2}\right)\left(\partial_{v}h_{1}-\partial_{v}h_{2}\right)
\]
so using that $h_{1},h_{2}\in B\left(b_{0}G^{-4}\right)\subset\mathcal{\tilde{Z}}_{1,1/2,\kappa,\delta,\sigma}$
we have 
\begin{eqnarray*}
\left\Vert \mathcal{F}\left(h_{2}\right)-\mathcal{F}\left(h_{1}\right)\right\Vert _{2,3/2} & \leq & \left\Vert \frac{1}{2y_{h}^{2}}\left(\partial_{v}h_{1}+\partial_{v}h_{2}\right)\right\Vert _{0,0}\left\Vert \partial_{v}h_{1}-\partial_{v}h_{2}\right\Vert _{2,3/2}\\
 & \leq & KG^{3/2}\left\Vert \frac{1}{2y_{h}^{2}}\left(\partial_{v}h_{1}+\partial_{v}h_{2}\right)\right\Vert _{0,1/2}\left\llbracket h_{1}-h_{2}\right\rrbracket _{1,1/2}\\
 & \leq & KG^{-5/2}\left\llbracket h_{1}-h_{2}\right\rrbracket _{1,1/2},
\end{eqnarray*}
and contractivity follows from Lemma \ref{lem:Inverse operator lemma}
(enlarging $G$ if necessary).

To obtain \eqref{eq:difference with potential} we notice that  
\begin{eqnarray*}
\left\Vert T_{1}-L_{1}^{u}\right\Vert _{1,1/2} & = & \left\Vert \mathcal{G}\circ\left(\mathcal{F}\left(T_{1}\right)-\mathcal{F}\left(0\right)\right)\right\Vert _{1,1/2}\\
 & \leq & \left\llbracket \mathcal{G}\circ\left(\mathcal{F}\left(T_{1}\right)-\mathcal{F}\left(0\right)\right)\right\rrbracket _{1,1/2}\\
 & \leq & KG^{-5/2}\left\llbracket T_{1}\right\rrbracket _{1,1/2}\leq KG^{-13/2}.
\end{eqnarray*}
\end{proof}

Since the parametrization \eqref{eq:H-J parametrization}  becomes singular at $v=0$, in the next section we look for a new parametrization of the unstable manifold which is regular at $v=0$ and therefore allows us to extend it  across $v=0$.

\subsection{Analytic continuation of the solution to the domain $D_{\kappa,\delta}^{\text{flow}}$}

In order to measure the distance between the stable and unstable manifolds
we need them to be defined in a common domain. However, a parametrization of the form
\[
\varGamma\left(v,\xi\right)=\left(\begin{array}{c}
\tilde{r}\left(v,\xi\right)\\
\tilde{y}\left(v,\xi\right)
\end{array}\right)=\left(\begin{array}{c}
\tilde{r}_{h}\left(v\right)\\
\frac{1}{\tilde{y}_{h}\left(v\right)}\partial_{v}T^{u}
\end{array}\right)
\]
becomes undefined at $v=0$. To avoid this difficulty we look for a different parametrization of the unstable manifold in the domain $D_{\rho,\kappa,\delta}$ \eqref{eq:Domain D rho kappa delta}
which does not contain $v=0$ and then extend it by the flow. In order to proceed, we introduce the
Banach spaces
\begin{equation}
\mathcal{Y}_{\mu,\rho,\kappa,\delta,\sigma}=\left\{ h:D_{\rho,\kappa,\delta}\times\mathbb{T}_{\sigma}\rightarrow\mathbb{C}\colon h\;\text{is real analytic},\:\left\Vert h\right\Vert _{\mu}<\infty\right\} \label{eq: Banach spaces second part}
\end{equation}
where 
\begin{equation}
\left\Vert h\right\Vert _{\mu}=\sum_{l\in\mathbb{Z}}\left\Vert h^{\left[l\right]}\right\Vert _{\mu}
\end{equation}
and
\begin{equation}
\left\Vert h^{\left[l\right]}\right\Vert _{\mu}=\sup_{v\in D_{\rho,\kappa,\delta}}\left|\left(v^{2}+1/9\right)^{\mu}h^{\left[l\right]}\left(v\right)\right|
\end{equation}
and the analogues of \eqref{eq: banach space derivative} 
\[
\mathcal{\tilde{Y}}_{\mu,\rho,\kappa,\delta}=\left\{ h:D_{\rho,\kappa,\delta}\times\mathbb{T}_{\sigma}\rightarrow\mathbb{C}\colon h\;\text{is real analytic},\:\left\llbracket h\right\rrbracket _{\mu}<\infty\right\} 
\]
with 
\[
\left\llbracket h\right\rrbracket _{\mu}=\left\Vert h\right\Vert _{\mu}+\left\Vert \partial_{v}h\right\Vert _{\mu+1}.
\]
\begin{rem}
\label{rem:Remark on the Banach spaces}Throughout this section we
will work on different domains $D_{\rho,\kappa,\delta}$, $D_{\kappa,\delta}^{\text{flow}}$
and $\tilde{D}_{\kappa,\delta}$ (the latter is defined in \eqref{eq:boomerang domain tilde}). We will denote by $\mathcal{Y}_{\mu,\kappa,\delta}$
the analogue to the Banach spaces  \eqref{eq: Banach spaces second part} associated to the domain $\tilde{D}_{\kappa,\delta}$, and by  $\mathcal{Y}_{\mu,\kappa,\delta}^{\text{flow}}$ the analogues for domain $D_{\kappa,\delta}^{\text{flow}}$ \eqref{eq:flow domain} (in this case for vectorial functions since we will work with vector fields on the plane).
\end{rem}

\subsubsection{From Hamilton-Jacobi parametrizations to parametrizations invariant
by the flow}

We look for a change of variables of the form $\text{\ensuremath{\mathrm{Id}}}+g:\left(v,\xi\right)\mapsto\left(v+g\left(v,\xi\right),\xi\right)$
such that 
\begin{equation}\label{eq:gammahat}
\hat{\varGamma}\left(v,\xi\right)=\varGamma\circ\left(\text{Id}+g\right)\left(v,\xi\right)
\end{equation}
satisfies 
\[
\phi_{s}\left(\hat{\varGamma}\left(v,\xi\right)\right)=\hat{\varGamma}\left(v+s,\xi+G^{3}s\right).
\]
Denoting by $X$ the vector field generated by the Hamiltonian \eqref{eq:rescaled_hamiltonian},
this equation is equivalent to 
\begin{equation}
X\circ\hat{\varGamma}=\mathcal{L}\left(\hat{\varGamma}\right),\label{eq: flow param equation}
\end{equation}
which we can rewrite as 
\begin{equation}
\mathcal{L}\left(g\right)\left(v,\xi\right)=\mathcal{F}\circ\left(\mathrm{Id}+g\right)\left(v,\xi\right)\qquad\mathcal{\text{where\qquad}F}=\frac{1}{y_{h}^{2}}\partial_{v}T_{1}\label{eq:change of param eq}
\end{equation} 
and $\mathcal{L}$ stands for the differential operator \eqref{eq:linear operator L}. As before we transform \eqref{eq:change of param eq} into a fixed
point equation. Thus, we introduce the inverse operator 
\[
\tilde{\mathcal{G}}\left(h\right)=\sum_{l\in\mathbb{Z}}\tilde{\mathcal{G}}\left(h\right)^{[l]}e^{il\xi}
\]
where 
\begin{align}
\tilde{\mathcal{G}}\left(h\right)^{[l]} & =\int_{v_{1}}^{v}e^{ilG^{3}\left(t-v\right)}h^{[l]}\left(t\right)\text{d}t\nonumber \\
\tilde{\mathcal{G}}\left(h\right)^{[0]} & =\int_{-\rho}^{v}h^{[l]}\left(t\right)\text{d}t\label{eq:inverse operator 2}\\
\tilde{\mathcal{G}}\left(h\right)^{[l]} & =\int_{\bar{v}_{1}}^{v}e^{ilG^{3}\left(t-v\right)}h^{[l]}\left(t\right)\text{d}t.\nonumber 
\end{align}
and $v_{1},\bar{v}_{1}$ are the top and bottom points of the domain
$D_{\rho,\kappa,\delta}$ defined in equation \eqref{eq:Domain D rho kappa delta}.
The following lemma is proved as Lemma 5.5 in \cite{Guardia2010}.
\begin{lem}
\label{lem:Inverse operator lemma 2}The operator $\mathcal{\tilde{G}}$
defined on \ref{eq:inverse operator 2} satisfies the following properties.

i) For any $\mu\geq0$, $\mathcal{\tilde{G}}:\mathcal{Y}_{\mu,\rho,\kappa,\delta,\sigma}\rightarrow\mathcal{Y}_{\mu,\rho,\kappa,\delta,\sigma}$
is well defined, linear and satisfies $\mathcal{L}\circ\mathcal{\tilde{G}}=\text{Id}.$

ii) If $h\in\mathcal{Y}_{\mu,\rho,\kappa,\delta,\sigma}$ for some
$\mu>1,$ then 
\begin{equation}
\left\Vert \mathcal{\tilde{G}}\left(h\right)\right\Vert _{\mu-1}\leq K\left\Vert h\right\Vert _{\mu}.\label{eq:inverse bound-1}
\end{equation}

iii) If $h\in\mathcal{Y}_{\mu,\rho,\kappa,\delta,\sigma}$ for some
$\mu\geq1,$ then 
\begin{equation}
\left\Vert \partial_{v}\mathcal{\tilde{G}}\left(h\right)\right\Vert _{\mu}\leq K\left\Vert h\right\Vert _{\mu}.\label{eq:inverse derivative bound-1}
\end{equation}
\end{lem}

Therefore, solutions of \eqref{eq:change of param eq} are also fixed
points of 
\begin{equation}
g=\tilde{\mathcal{G}}\circ\mathcal{F}\circ\left(\mathrm{Id}+g\right).\label{eq:Fixed point for g}
\end{equation}
We state two technical lemmas which will be useful for dealing with
compositions of functions and are deduced from the proofs of Lemmas
5.14 and 5.15 in \cite{Guardia2016}.
\begin{lem}
Fix constants $\delta'<\delta,$ $\rho'<\rho$ , $\kappa'>\kappa$
and take $h\in\mathcal{Y}_{\mu,\rho,\kappa,\delta,\sigma}.$ Then,
$\partial_{v}h\in\mathcal{Y}_{\mu,\rho',\kappa',\delta',\sigma}$
and satisfy 
\[
\left\Vert \partial_{v}h\right\Vert _{\mu}\leq\frac{G^{3}}{\left(\kappa'-\kappa\right)}\left(\frac{\kappa^{'}}{\kappa}\right)^{\mu}\left\Vert h\right\Vert _{\mu}.
\]
\end{lem}

\begin{lem}
\label{lem:Tech lemma composition}Fix constants $\rho'<\rho$, $\delta'<\delta,$
$\kappa'>\kappa+1$ and $\sigma'<\sigma$. Then,

i) If $h\in\mathcal{Y}_{\mu,\rho,\kappa,\delta,\sigma}$ and $g\in B\left(G^{-3}\right)\subset\mathcal{Y}_{\mu,\rho',\kappa',\delta',\sigma'}$
we have that $\tilde{h}=h\circ\left(\mathrm{Id}+g\right)\in\mathcal{Y}_{\mu,\rho',\kappa',\delta',\sigma'}$
and 
\[
\left\Vert \tilde{h}\right\Vert _{\mu}\leq\left(\frac{\kappa^{'}}{\kappa}\right)^{\mu}\left\Vert h\right\Vert _{\mu}.
\]

ii) Moreover if $g_{1},g_{2}\in B\left(G^{-3}\right)\subset\mathcal{Y}_{\mu,\rho',\kappa',\delta',\sigma'}$,
then $f=h\circ\left(\mathrm{Id}+g_{1}\right)-h\circ\left(\text{\ensuremath{\mathrm{Id}}}+g_{2}\right)$
satisfies 
\[
\left\Vert f\right\Vert _{\mu}\leq\frac{G^{3}}{\left(\kappa'-\kappa\right)}\left(\frac{\kappa'}{\kappa}\right)^{\mu}\left\Vert h\right\Vert _{\mu}\left\Vert g_{1}-g_{2}\right\Vert _{0,0}.
\]
\end{lem}

\begin{thm}
\label{thm:Existence change of variables} Let $\delta,\kappa$ and
$\sigma$ be the constants given by Theorem \ref{thm:Existence infinity manifold close to infinity}.
Let $\rho_{1}<\rho,$ $\delta_{1}<\delta,$ $\sigma_{1}<\sigma$ and
$\kappa_{1}>\kappa.$ Then, for $G$ big enough, there exist a function
$g\in\mathcal{Y}_{0,\rho_{1},\kappa_{1},\delta_{1},\sigma_{1}}$ satisfying
\[
\left\Vert g\right\Vert _{0}\leq b_{1}G^{-7/2}
\]
for $b_{1}>0$ independent of $G$ and such that 
\[
\hat{\varGamma}=\varGamma\circ\left(\mathrm{Id}+g\right)
\]
satisfies \eqref{eq: flow param equation}. 
\end{thm}

\begin{proof}
To find $g$ we solve the fixed point equation \eqref{eq:Fixed point for g}. For that, we take $g\in B\left(KG^{-7/2}\right)\subset\mathcal{Y}_{0,\rho_{1},\kappa_{1},\delta_{1},\sigma_{1}}$,
with $K$ a constant independent of $G.$ Then by Lemma \ref{lem:Tech lemma composition}
and using the estimate for $\partial_{v}T_{1}$ obtained in Theorem
\ref{thm:Existence infinity manifold close to infinity} we have 
\begin{eqnarray*}
\left\Vert \mathcal{F}\circ\left(\mathrm{Id}+g\right)\right\Vert _{1/2} & \leq & \left(\frac{\kappa_{1}}{\kappa}\right)^{1/2}\left\Vert \mathcal{F}\right\Vert _{1/2}\\
 & \leq & \left(\frac{\kappa_{1}}{\kappa}\right)^{1/2}KG^{-4}\\
 & \leq & KG^{-4}
\end{eqnarray*}
where $K$ is a constant depending only on the reduction of the domain.
From here it is clear using Lemma \ref{lem:Inverse operator lemma 2}
that the map $\tilde{\mathcal{G}}\circ\mathcal{F}\circ\left(\text{Id}+g\right):B\left(KG^{-7/2}\right)\subset\mathcal{Y}_{0,\rho_{1},\kappa_{1},\delta_{1},\sigma_{1}}\rightarrow\mathcal{Y}_{0,\rho_{1},\kappa_{1},\delta_{1},\sigma_{1}}$
is well defined. Moreover, we obtain that 
\begin{equation}
\begin{split}
\left\Vert \tilde{\mathcal{G}}\circ\mathcal{F}\circ\left(\text{Id}+g\right)_{|g=0}\right\Vert _{0}&\leq KG^{1/2}\left\Vert \tilde{\mathcal{G}}\circ\mathcal{F}\circ\left(\text{Id}+g\right)_{|g=0}\right\Vert _{1/6}\\
&\leq KG^{1/2}\left\Vert \mathcal{F}\circ\left(\text{Id}+g\right)_{|g=0}\right\Vert _{7/6}\\
&\leq KG^{1/2}\left\Vert \mathcal{F}\circ\left(\text{Id}+g\right)_{|g=0}\right\Vert _{1/2}\\
&\leq b_{1}G^{-7/2}
\end{split}
\end{equation}
for some $b_{1}$ independent of $G.$ It only remains to show that
the map $\tilde{\mathcal{G}}\circ\mathcal{F}\circ\left(\text{Id}+g\right)$
is contractive in a neighbourhood of the origin. Take $g_{1},g_{2}\in B\left(b_{1}G^{-7/2}\right)\subset\mathcal{Y}_{0,\rho_{1},\kappa_{1},\delta_{1},\sigma_{1}}$,
using again Lemma \ref{lem:Tech lemma composition} we have that 
\[
\left\Vert \mathcal{F}\circ\left(\text{Id}+g_{1}\right)-\mathcal{F}\circ\left(\text{Id}+g_{2}\right)\right\Vert _{1/2}\leq\tilde{K}G^{-1/2}\left\Vert g_{1}-g_{2}\right\Vert _{0}.
\]
Direct application of Lemma \ref{lem:Inverse operator lemma} yields
\[
\left\Vert \mathcal{\tilde{G}}\left(\mathcal{F}\circ\left(\text{Id}+g_{1}\right)-\mathcal{F}\circ\left(\text{Id}+g_{2}\right)\right)\right\Vert _{0}\leq\tilde{K}G^{-1/2}\left\Vert g_{1}-g_{2}\right\Vert _{0},
\]
so for $G$ big enough the map $g\mapsto\tilde{\mathcal{G}}\circ\mathcal{F}\circ\left(\text{Id}+g\right)$
is contractive on $B\left(b_{1}G^{-7/2}\right)\subset\mathcal{Y}_{0,\rho_{1},\kappa_{1},\delta_{1},\sigma_{1}}$
and the proof is completed. 
\end{proof}

\subsubsection{Analytic extension of the unstable manifold by the flow parametrization}

Now we perform the analytic continuation of the parametrization \eqref{eq:gammahat}
given by Theorem \ref{thm:Existence infinity manifold close to infinity}
to the domain $D_{\kappa,\delta}^{\mathrm{flow}}$ defined in \eqref{eq:flow domain}
using the flow of the Hamiltonian \eqref{eq:rescaled_hamiltonian}.
Notice that since the domain $D_{\kappa,\delta}^{\mathrm{flow}}$ is bounded and at distance of order
$\mathcal{O}\left(1\right)$ with respect to the singularities all
norms $\left\Vert h\right\Vert _{\mu}$ are equivalent, therefore
it will suffice to get estimates on the norm $\left\Vert h\right\Vert _{0}$.

Write $\hat{\varGamma}=\hat{\varGamma}_{0}+\hat{\varGamma}_{1}$, where 
\begin{equation}
\hat{\varGamma}_{0}\left(v,\xi\right)=\varGamma_{0}\circ\left(\mathrm{Id}+g\right)\left(v,\xi\right)\qquad\varGamma_{0}\left(v\right)=\left(\tilde{r}_h\left(v\right),\tilde{y}_h\left(v\right)\right).
\end{equation}
Then, the equation \eqref{eq: flow param equation} that defines this extension
is rewritten as
\begin{equation}
\mathcal{\hat{L}}\left(\hat{\varGamma}_{1}\right)=\mathcal{\hat{F}}\left(\hat{\varGamma}_{1}\right)\label{eq:flow eq 2}
\end{equation}
where 
\[
\begin{split}\hat{\mathcal{L}}\left(h\right)= & \mathcal{L}\left(h\right)-DX_{0}\left(\hat{\varGamma}_{0}\right)h\\
\mathcal{\hat{F}}\left(h\right)= & X_{0}\left(\hat{\varGamma}_{0}+h\right)-X_{0}\left(\hat{\varGamma}_{0}\right)-DX_{0}\left(\hat{\varGamma}_{0}\right)h+X_{1}\left(\hat{\varGamma}_{0}+h\right).
\end{split}
\]
Denote by $\varPsi\left(v\right)$ the fundamental matrix of the linear
system 
\[
\dot{z}\left(v\right)=DX_{0}\left(\varGamma_{0}\left(v,\xi\right)\right)z\left(v\right),\qquad v\in D_{\kappa,\delta}^{\mathrm{flow}}.
\]
Then, equation \eqref{eq:flow eq 2}, together with a suitable initial condition $\hat{\varGamma}_{h}$,  can be reformulated as the fixed point equation 
\begin{equation}
\hat{\varGamma}_{1}=\hat{\varGamma}_{h}+\hat{\mathcal{G}}\circ\hat{\mathcal{F}}\left(\hat{\varGamma}_{1}\right),\label{eq: Fixed point equation Sect 5.2.2}
\end{equation}
where 
\begin{eqnarray*}
\hat{\varGamma}_{h} & = & \sum_{l>0}\varPsi\left(v\right)\varPsi^{-1}\left(v_{1}\right)\hat{\varGamma}_{1}^{[l]}\left(v_{1}\right)e^{ilG^{3}\left(v_{1}-v\right)}e^{il\xi}\\
 &  & +\sum_{l<0}\varPsi\left(v\right)\varPsi^{-1}\left(\bar{v}_{1}\right)\hat{\varGamma}_{1}^{[l]}\left(\bar{v}_{1}\right)e^{ilG^{3}\left(\bar{v}_{1}-v\right)}e^{il\xi}\\
 &  & +\varPsi\left(v\right)\varPsi^{-1}\left(-\rho_{1}\right)\hat{\varGamma}_{1}^{[0]}\left(-\rho_{1}\right)
\end{eqnarray*}
is the solution of the homogeneous equation $\hat{\mathcal{L}}\left(h\right)=0$ (observe that since $v_1,\bar{v}_1,-\rho_1$ are contained in $D_{\rho,\kappa,\delta}$, the terms $\hat{\varGamma}_{1}^{[l]}\left(v_{1}\right), \hat{\varGamma}_{1}^{[l]}\left(\bar{v}_{1}\right)$ and $\hat{\varGamma}_{1}^{[l]}\left(-\rho_{1}\right)$ are already defined)
and 
\[
\hat{\mathcal{G}}\left(h\right)=\varPsi\tilde{\mathcal{G}}\left(\varPsi^{-1}h\right)
\]
is a right inverse operator. Notice that since $DX\left(\hat{\varGamma}_{0}\left(v,\xi\right)\right)$
is continuous and $D_{\kappa,\delta}^{\mathrm{flow}}$ is a compact
domain at distance $\mathcal{O}\left(1\right)$ from the singularities,
we have that there exists $K>0$ such that 
\begin{equation}
\sup_{v\in D_{\kappa,\delta}^{\text{flow}}}\max\left\{ \left\Vert \varPsi\right\Vert _{0},\left\Vert \varPsi^{-1}\right\Vert _{0}\right\} \leq K,\label{eq:bound fundamental matrix}
\end{equation}
in the matrix norm associated to the usual vector norm in $\mathbb{C}^{2}.$
\begin{lem}
\label{lem:Technical lemma for analytic extension}Assume $h,\tilde{h}\in B\left(KG^{-4}\right)\subset\mathcal{Y}_{0,\kappa_{1},\delta_{1},\sigma_{1}}^{\mathrm{flow}}$
for some $K>0.$ Then there exists $K^{'}>0$ such that

i) Defining $Y\left(h\right)=X_{0}\left(\hat{\varGamma}_{0}+h\right)-X_{0}\left(\hat{\varGamma}_{0}\right)-DX_{0}\left(\hat{\varGamma}_{0}\right)h$
we have that $Y\left(h\right)\in\mathcal{Y}_{0,\kappa_{1},\delta_{1},\sigma_{1}}^{\mathrm{flow}}.$
and 
\[
\left\Vert Y\left(h\right)\right\Vert _{0}\leq K^{'}G^{-4},
\]

ii) $\ X_{1}\left(\hat{\varGamma}_{0}+h\right)\in\mathcal{Y}_{0,\kappa_{1},\delta_{1},\sigma_{1}}^{\mathrm{flow}}.$
with $\left\Vert X_{1}\left(\hat{\varGamma}_{0}+h\right)\right\Vert _{0}\leq K^{'}G^{-4},$

iii)$\ \left\Vert Y\left(h\right)-Y\left(\tilde{h}\right)\right\Vert _{0}\leq K^{'}G^{-4}\left\Vert h-\tilde{h}\right\Vert _{0},$

iv)$\ \left\Vert X_{1}\left(\hat{\varGamma}_{0}+h\right)-X_{1}\left(\hat{\varGamma}_{0}+\tilde{h}\right)\right\Vert _{0}\leq K^{'}G^{-4}\left\Vert h-\tilde{h}\right\Vert _{0}.$ 
\end{lem}

\begin{proof}
The proof follows from the mean value theorem together with the straightforward
bounds

\[
\left\Vert DX_{0}\left(\hat{\varGamma}_{0}\right)\right\Vert _{0}\leq K^{'}\quad\left\Vert X_{1}\left(\hat{\varGamma}_{0}\right)\right\Vert _{0}\leq K^{'}G^{-4}\quad\left\Vert DX_{1}\left(\hat{\varGamma}_{0}\right)\right\Vert _{0}\leq K^{'}G^{-4}.
\]
\end{proof}
\begin{prop}
\label{prop:Analytic continuation existence}Let $\kappa_{1},\ \delta_{1}$
and $\sigma_{1}$ be the constants considered in Theorem \ref{thm:Existence change of variables}.
Then, there exists $b_{2}>0$ such that if $G$ is large enough, the
fixed point equation \eqref{eq: Fixed point equation Sect 5.2.2}
has a unique solution $\hat{\varGamma}_{1}\in B\left(b_{2}G^{-4}\right)\subset\mathcal{Y}_{0,\kappa_{1},\delta_{1},\sigma_{1}}^{\mathrm{flow}}.$ 
\end{prop}

\begin{proof}
As $v_{1},\bar{v}_{1},\rho_{1}\in D_{\rho_{1},\kappa_{1},\delta_{1}}$
we have that $\hat{\varGamma}_{h}\in\mathcal{Y}_{0,\rho_{1},\kappa_{1},\delta_{1}}$
with 
\[
\left\Vert \hat{\varGamma}_{h}\right\Vert _{0}\leq KG^{-4}.
\]
We claim using Lemma \ref{lem:Technical lemma for analytic extension}
that the map $\hat{\mathcal{K}}:h\mapsto\varGamma_{h}+\hat{\mathcal{G}}\circ\mathcal{\hat{\mathcal{F}}}\left(h\right)$
is well defined from $B\left(KG^{-4}\right)\subset\mathcal{Y}_{0,\kappa_{1},\delta_{1},\sigma_{1}}^{\mathrm{flow}}$
to $\mathcal{Y}_{0,\kappa_{1},\delta_{1},\sigma_{1}}^{\mathrm{flow}}..$
Moreover, we see from the estimate \eqref{eq:bound fundamental matrix}
for the fundamental matrix $\varPsi\left(v\right)$ that there exists
$b_{2}$ such that 
\[
\left\Vert \mathcal{\hat{K}}\left(0\right)\right\Vert _{0}=\left\Vert \varGamma_{h}+\hat{\mathcal{G}}\left(X_{1}\circ\varGamma_{0}\right)\right\Vert _{0}\leq\frac{b_{2}}{2}G^{-4}.
\]
Finally, from Lemma \ref{lem:Technical lemma for analytic extension},
we conclude that for $G$ big enough $\hat{\mathcal{K}}$ is Lipschitz
in $B\left(b_{2}G^{-4}\right)\subset\mathcal{Y}_{0,\kappa_{1},\delta_{1},\sigma_{1}}^{\mathrm{flow}}$
with Lipschitz constant $KG^{-4}.$ 
\end{proof}

\subsubsection{From flow parametrization to Hamilton-Jacobi parametrization}

Now that we have extended the parametrization \eqref{eq:gammahat} across $v=0$, the
next step is to come back to the Hamilton-Jacobi parametrization \eqref{eq:H-J parametrization} so
we have both stable and unstable manifolds parametrized as graphs of the rofm
$\left(\tilde{r}_{h}\left(v\right),\tilde{y}^{u,s}\left(v,\xi\right)\right)$
and we can easily measure the distance between them.

We look for a change of variables of the form $\mathrm{Id}+f$ such
that 
\begin{equation}
\pi_{1}\circ\hat{\varGamma}\circ\left(\mathrm{Id}+f\right)\left(v,\xi\right)=\tilde{r}_{h}\left(v\right)\label{eq:inverse change 1}
\end{equation}
in the domain 
\begin{equation}\label{eq:boomerang domain tilde}
\tilde{D}_{\kappa_{1},\delta_{1}}= D_{\kappa_1,\delta_1}^{flow}\cap D_{\kappa_1,\delta_1},
\end{equation}
where $D_{\kappa_1,\delta_1}^{flow}, D_{\kappa_1,\delta_1}$ are the domains defined in \eqref{eq:flow domain} and \eqref{eq:boomerang domain}. Therefore, in
$D_{\rho_{1},\kappa_{1},\delta_{1}}^{u}\cap\tilde{D}_{\kappa_{1},\delta_{1}}$
the change $\mathrm{Id}+f$ is the inverse of the change $\mathrm{Id}+g$
obtained in Theorem \ref{thm:Existence change of variables}. We will see that this change of variables is unique under certain conditions, therefore, once we have $f$, the second component of the unstable manifold
is given by 
\begin{equation}
\pi_{2}\circ\hat{\varGamma}_1\circ\left(\mathrm{Id}+f\right)\left(v,\xi\right)=\frac{1}{y_{h}\left(v\right)}\partial_{v}T_1.\label{eq:invariant manifold hj}
\end{equation}
Using the properties of the unperturbed solution, i.e. $\pi_{1}\circ\varGamma_{0}\left(v,\xi\right)=\tilde{r}_{h}\left(v\right)$,
we can write equation \eqref{eq:inverse change 1} as 
\[
f=\mathcal{P}\left(f\right)
\]
where 
\[
\mathcal{P}\left(f\right)=\frac{-1}{y_{h}\left(v\right)}\left(\tilde{r}_{h}\left(v+f\left(v,\xi\right)\right)-\tilde{r}_{h}\left(v\right)-\tilde{y}_{h}\left(v\right)f\left(v,\xi\right)-\pi_{1}\circ\varGamma_{1}\circ\left(\text{Id}+f\right)\left(v,\xi\right)\right).
\]
\begin{prop}
\label{prop:Existence inverse change}Consider the constants $\kappa_{1},\ \delta_{1}$
and $\sigma_{1}$ given by Proposition \ref{prop:Analytic continuation existence}
and any $\kappa_{2}>\kappa_{1},$ $\delta_{2}<\delta_{1},$ $\sigma_{2}<\sigma_{1}.$
Then,

i) There exists $b_{3}>0$ such that for $G$ 
large enough, the operator $\mathcal{P}$ has a unique fixed point
$f\in\mathcal{Y}_{0,\kappa_{2},\delta_{2},\sigma_{2}}$with 
\[
\left\Vert f\right\Vert _{0}\leq b_{3}G^{-4}.
\]

ii) Equation \eqref{eq:invariant manifold hj} defines the graph of
the unstable manifold wich can be written as $T^{u}=T_{0}+T_{1}^{u}$
where $T_{1}^{u}$ satisfies 
\[
\left\Vert \partial_{v}T_{1}^{u}\right\Vert _{0}\leq KG^{-4}.
\]
\end{prop}

\begin{proof}
For the first part we observe that, for $f_{2},f_{1}\in B\left(KG^{-4}\right)\subset\mathcal{Y}_{0,\kappa_{2},\delta_{2},\sigma_{2}},$
\begin{align*}
\left|\tilde{r}_{h}\left(v+f_{2}\right)-\tilde{r}_{h}\left(v+f_{1}\right)-\tilde{y}_{h}\left(f_{2}-f_{1}\right)\right| & \leq K\left|f_{2}^{2}-f_{1}^{2}\right|\\
 & \leq KG^{-4}\left|f_{2}-f_{1}\right|.
\end{align*}
Then, from Lemma \ref{lem:Tech lemma composition} and the fact and
$\left\Vert \hat{\varGamma}_{1}^{u}\right\Vert _{0}\leq KG^{-4}$ we deduce
that 
\[
\left|\mathcal{P}\left(f_{2}\right)-\mathcal{P}\left(f_{1}\right)\right|\leq KG^{-4}\left|f_{2}-f_{1}\right|,
\]
i.e. $\mathcal{P}\left(f\right)$ is a contractive mapping on $B\left(b_{3}G^{-4}\right)\subset\mathcal{Y}_{0,\kappa_{2},\delta_{2},\sigma_{2}}$
for some $b_{3}>0$ so there exist a unique $f\in B\left(b_{3}G^{-4}\right)\subset\mathcal{Y}_{0,\kappa_{2},\delta_{2},\sigma_{2}}$
solving $f=\mathcal{P}\left(f\right).$

For the second part we have from equation \eqref{eq:invariant manifold hj}
that 
\[
\pi_{2}\circ\hat{\varGamma}_{1}\circ\left(\mathrm{Id}+f\right)\left(v,\xi\right)=\frac{1}{y_{h}\left(v\right)}\partial_{v}T_{1}.
\]
Therefore, 
\begin{eqnarray*}
\left\Vert \partial_{v}T_{1}\right\Vert _{0,0} & \leq & K\left\Vert \frac{1}{y_{h}\left(v\right)}\partial_{v}T_{1}\right\Vert _{0}\\
 & = & K\left\Vert \pi_{2}\circ\hat{\varGamma}_{1}\circ\left(\mathrm{Id}+f\right)\right\Vert _{0}\\
 & \leq & K\left\Vert \hat{\varGamma}_{1}\circ\left(\mathrm{Id}+f\right)\right\Vert _{0}\\
 & \leq & K\left\Vert \hat{\varGamma}_{\text{1}}\right\Vert _{0}\leq KG^{-4},
\end{eqnarray*}
where we have used Lemma \ref{lem:Tech lemma composition} and the
estimate for $\left\Vert \hat{\varGamma}_{\text{1}}\right\Vert _{0}$
obtained in Proposition \ref{prop:Analytic continuation existence}. 
\end{proof}
We sum up the results obtained in this section in the following theorem.
\begin{thm}
\label{thm: Final Theorem sect 3}Let $\kappa_{2},$ $\delta_{2}$
and $\sigma_{2}$ the constants given by Proposition \ref{prop:Existence inverse change}.
Then, for $G$ big enough there exist real analytic functions $T_{1}^{u,s}$
defined in $D_{\kappa_{2},\delta_{2}}$ which are solutions of equation
\eqref{eq:eq T1} and satisfy 
\[
\left\Vert \partial_{v}T_{1}^{u,s}\right\Vert _{3/2}\leq b_{4}G^{-4}
\]
for a certain $b_{4}>0$ independent of $G.$ 
\end{thm}

\begin{proof}
For the stable manifold, the result was obtained in Theorem \ref{thm:Existence infinity manifold close to infinity}
since $D_{\kappa_{2},\delta_{2}}\subset D_{\kappa,\delta}^{\infty,s}$.
For the unstable manifold, using that $D_{\kappa_{2},\delta_{2}}\subset D_{\kappa,\delta}^{\infty,u}\cup \tilde{D}_{\kappa_{2},\delta_{2}}$
the result follows from the combination of Theorem \ref{thm:Existence infinity manifold close to infinity}
and Proposition \ref{prop:Existence inverse change}. 
\end{proof}

\section{\label{sec:The-difference-between}The difference between the manifolds}

Once we have obtained the parametrization of the invariant manifolds
in the common domain $D_{\kappa,\delta}$ defined in \eqref{fig:Domain Boomerang},
the next step is to study their difference. To this end we define
\begin{equation}
\tilde{\varDelta}\left(v,\xi\right)=T^{s}\left(v,\xi\right)-T^{u}\left(v,\xi\right).\label{eq:Difference between manifolds def.}
\end{equation}
Substracting equation \eqref{eq:eq T1} for $T_{1}^{s}$ and $T_{1}^{u}$
one obtains that 
\[
\tilde{\varDelta}\in\mathrm{Ker}\tilde{\mathcal{L}}
\]
where $\mathcal{\tilde{L}}$ is the differential operator 
\[
\tilde{\mathcal{L}}=\left(1+A\left(v,\xi\right)\right)\partial_{v}-G^{3}\partial_{\xi}
\]
with 
\begin{equation}
A\left(v,\xi\right)=\frac{1}{2\tilde{y}_{h}^{2}}\left(\partial_{v}T_{1}^{s}-\partial_{v}T_{1}^{u}\right).\label{eq:Definition A}
\end{equation}
To obtain exponentially small bounds on the difference between the
invariant manifolds we will look for a close to identity change of
variables $\left(v,\xi\right)=\left(w+C\left(w,\xi\right),\xi\right)$
such that the function 
\begin{equation}\label{eq: delta periodic}
\varDelta\left(w,\xi\right)=\tilde{\varDelta}\left(w+C\left(w,\xi\right),\xi\right)\qquad \left(w,\xi\right)\in D_{\kappa,\delta}\times\mathbb{T}_{\sigma},
\end{equation}
satisfies 
\[
\mathcal{\varDelta\in\mathrm{Ker}L}
\]
 where $\mathcal{L}$ is the differential operator defined in \eqref{eq:linear operator L}.
The condition $\varDelta\in\mathrm{Ker}\mathcal{L}$ implies that $\varDelta=f\left(\xi-G^{3}w\right)$. Therefore, since  $\varDelta$ is periodic in $\xi$ it must be periodic in $w$. Since $\varDelta$ is real analytic and bounded
in a strip that reaches up to points $\mathcal{O}\left(G^{-3}\right)$
close to the singularities the exponentially small bound for $\left|\varDelta\left(w,\xi\right)\right|$
where $w\in\mathbb{R}$ comes straightforward by a classical argument
(see Lemma \ref{lem:Exponentially small technical lemma} below).
We devote the rest of the section to make this rigorous.

\subsection{Straightening the operator $\mathcal{\tilde{L}}$}

As we did in the previous sections we introduce the Banach spaces
\[
\mathcal{Q}_{\mu,\rho,\kappa,\delta,\sigma}=\left\{ h:D_{\kappa,\delta}\times\mathbb{T}_{\sigma}\rightarrow\mathbb{C}\colon h\;\text{is real analytic},\:\left\Vert h\right\Vert _{\mu}<\infty\right\} 
\]
where 
\[
\left\Vert h\right\Vert _{\mu}=\sup_{v\in D_{\kappa,\delta}}\left|\left(v^{2}+1/9\right)^{\mu}h\left(v\right)\right|.
\]
\begin{thm}
\label{thm:change of variables C}Let $\kappa_{2}$ and $\delta_{2}$
the constants defined in Theorem \ref{thm: Final Theorem sect 3}.
Let $\kappa_{3}>\kappa_{2},$ $\delta_{3}<\delta_{2}$ and $\sigma_{3}<\sigma_{2}$
be fixed. Then, for $G$ big enough, there exists $C\in\mathcal{Q}_{0,\kappa_{3},\delta_{3},\sigma_{3}}$
such that the function 
\[
\varDelta\left(w,\xi\right)=\tilde{\varDelta}\left(w+C\left(w,\xi\right),\xi\right)
\]
satisfies that $\varDelta\in\mathrm{Ker}\mathcal{L}.$ Moreover, we
have that 
\[
\left\Vert C\right\Vert _{0}\leq b_{5}G^{-7/2}
\]
for a certain $b_{5}>0$ independent of $G.$ 
\end{thm}

\begin{proof}
Using the chain rule we obtain that the implication $\varDelta\in\mathrm{Ker}\mathcal{L}$
if and only if $\varDelta\in\mathrm{Ker}\mathcal{\tilde{L}}$, is
equivalent to finding $C$ satisfying 
\begin{align*}
\mathcal{L}\left(C\right) & =A_{|v=w+C\left(w\right)}\\
 & =A\circ\left(\mathrm{Id}+C\right),
\end{align*}
where $A\left(v,\xi\right)$ was defined in \eqref{eq:Definition A}.
We can rewrite this equation as a fixed point equation 
\[
C=\tilde{\mathcal{G}}\left(A\circ\left(\mathrm{Id}+C\right)\right),
\]
where $\tilde{\mathcal{G}}$ is the inverse operator defined in \eqref{eq:inverse operator 2}.
Using the bounds for $\partial_{v}T_{1}^{u,s}$ in Theorem \ref{thm: Final Theorem sect 3},
 the properties of the homoclinic orbit stated in Section \ref{sec:The-integrable-system}
, and Lemma \ref{lem:Tech lemma composition} for the composition,
we obtain that, for $C\in B\left(KG^{-4}\right)\subset\mathcal{Q}_{0,\kappa_{3},\delta_{3},\sigma_{3}},$
\[
\left\Vert A\circ\left(\mathrm{Id}+C\right)\right\Vert _{\text{1/2}}\leq K^{'}G^{-4}
\]
for some $K^{'}>0$ independent of $G.$ Hence, from Lemma \ref{lem:Tech lemma composition}
we observe that the map $C\mapsto\tilde{\mathcal{G}}\left(A\circ\left(\mathrm{Id}+C\right)\right)$
is well defined from $C\in B\left(KG^{-7/2}\right)\subset Q_{0,\kappa_{3},\delta_{3},\sigma_{3}}\rightarrow\mathcal{Q}_{0,\kappa_{3},\delta_{3},\sigma_{3}}.$
Moreover, we also get 
\[
\left\Vert \mathcal{\tilde{G}}\left(A\circ\left(\mathrm{Id}+C\right)_{|C=0}\right)\right\Vert _{0}\leq\frac{b_{5}}{2}G^{-7/2},
\]
for some $b_{5}$ independent of $G.$ Hence, it only remains to prove
that the map $C\mapsto\tilde{\mathcal{G}}\left(A\circ\left(\mathrm{Id}+C\right)\right)$
is contractive on the ball $B\left(b_{5}G^{-7/2}\right)\subset\mathcal{Q}_{0,\kappa_{3},\delta_{3,},\sigma_{3}}.$
Again by Lemma \ref{lem:Tech lemma composition} we have that if $C_{1},C_{2}\in B\left(b_{5}G^{-7/2}\right)\subset\mathcal{Q}_{0,\kappa_{3},\delta_{3},\sigma_{3}}$,
then 
\begin{eqnarray*}
\left\Vert A\circ\left(\mathrm{Id}+C_{2}\right)-A\circ\left(\mathrm{Id}+C_{1}\right)\right\Vert _{1/2} & \leq & KG^{3}\left\Vert A\right\Vert _{1/2}\left\Vert C_{2}-C_{1}\right\Vert _{0}\\
 & \leq & KG^{-1}\left\Vert C_{2}-C_{1}\right\Vert _{0},
\end{eqnarray*}
and contractivity follows from Lemma \ref{lem:Inverse operator lemma 2}
for $G$ big enough. 
\end{proof}

\subsection{Estimates for the difference between the invariant manifolds }

Now we exploit the fact that the function $\varDelta\left(w,\xi\right)$
defined in \eqref{eq: delta periodic} satisfies 
\[
\varDelta\in\mathrm{Ker}\mathcal{L}
\]
to get exponentially small bounds on the real line.
\begin{lem}
\label{lem:Exponentially small technical lemma}Let $h:D_{\kappa,\delta}\times\mathbb{T}_{\sigma}\mathbb{\rightarrow\mathbb{C}}$
be a real-analytic function such that $h\in\mathcal{Q}_{0,\kappa,\delta,\sigma}$
and $h\in\mathrm{Ker}\mathcal{L}.$ Then,

i) $h$ is of the form 
\[
h\left(w,\xi\right)=\sum_{l\in\mathbb{Z}}h^{[l]}\left(w\right)e^{il\xi}=\sum_{l\in\mathbb{Z}}\beta^{[l]}e^{il\left(\xi-G^{3}w\right)}.
\]

ii) the coefficients $\beta^{[l]}$ satisfy the bounds 
\[
\left|\beta^{[l]}\right|\leq\left\Vert h\right\Vert _{0}K^{\left|l\right|}e^{\frac{-\left|l\right|G^{3}}{3}}.
\]
\end{lem}

\begin{proof}
Since $h\in\mathrm{Ker}\mathcal{L}$ and is periodic in $\xi$, we have that each Fourier coefficient
$h^{[l]}$ satisfies 
\[
\frac{\text{d}}{\text{d}w}h^{[l]}+ilG^{3}h^{[l]}=0
\]
so it has to be 
\[
h^{[l]}\left(w\right)=\beta^{[l]}e^{-ilG^{3}w}
\]
for certain constants $\beta^{[l]}$. Moreover, evaluating this equality
at the top vertex $w_{2}=i\left(1/3-\kappa G^{-3}\right)$ of the
domain $D_{\kappa,\delta}$ for $l<0$ and at the bottom vertex $\bar{w}_{2}=i\left(1/3-\kappa G^{-3}\right)$
for $l>0$ we obtain that 
\begin{align*}
\left|\beta^{[l]}\right| & \leq\max\left\{ h^{[l]}\left(w_{2}\right),h^{[l]}\left(\bar{w}_{2}\right)\right\} e^{\frac{-\left|l\right|G^{3}}{3}}e^{\left|l\right|\kappa_{3}}\\
 & \leq\left\Vert h\right\Vert _{0}e^{\left|l\right|\kappa_{3}}e^{\frac{-\left|l\right|G^{3}}{3}}\\
 & \leq\left\Vert h\right\Vert _{0}K^{\left|l\right|}e^{\frac{-\left|l\right|G^{3}}{3}},
\end{align*}
for a constant $K$ independent of $G$ and $l$. Therefore, for $u\in\mathbb{R}\cap D_{\kappa,\delta}$ 
\[
\left|h^{[l]}\left(u\right)\right|=\left|\beta^{[l]}\right|\leq\left\Vert h\right\Vert _{0}K^{\left|l\right|}e^{\frac{-\left|l\right|G^{3}}{3}}.
\]
\end{proof}
Using this lemma we already have exponentially small bounds for $\varDelta\left(w,\xi\right).$
Nevertheless, our goal is to prove that the function $L$ defined
in \eqref{eq:Melnikov potential} is the main term in $\varDelta.$
Thus we study the function 
\[
\mathcal{E}\left(w,\xi\right)=\varDelta\left(w,\xi\right)-L\left(w,\xi\right).
\]
\begin{lem}
Consider the constants $\kappa_{3}\ \delta_{3}$ and $\sigma_{3}$
defined in Theorem \ref{thm:change of variables C}. Then, for $\left(w,\xi\right)\in\left(D_{\kappa_{3},\delta_{3}}\cap\mathbb{R}\right)\times\mathbb{T}$
we get 
\[
\left|\mathcal{E}\left(w,\xi\right)-E\right|\leq KG^{-\nicefrac{7}{2}}e^{\frac{-G^{3}}{3}}.
\]
where $E$ is a constant and 
\[
\left|\partial_{w}\mathcal{E}\right|\leq KG^{-\nicefrac{1}{2}}e^{\frac{-G^{3}}{3}}.
\]
\end{lem}

\begin{proof}
Notice that $L=L^{s}-L^{u}$ where $L^{*}=\mathcal{G}^{*}\left(V\right)$,
with $\mathcal{G}^{u,s}$ are the left inverse operators introduced
in \eqref{eq:inverse operator}. Then, it is clear that $\mathcal{L}\left(L\right)=0$
and we have that $\mathcal{E}\in\mathrm{Ker}\mathcal{L}.$ We bound
$\mathcal{E}$ in the domain $D_{\kappa,\delta}$ so that we can apply Lemma \ref{lem:Exponentially small technical lemma}.
We decompose $\mathcal{E}=\mathcal{E}_{1}^{s}-\mathcal{E}_{1}^{u}+\mathcal{E}_{2}$
where 
\begin{align*}
\mathcal{E}_{1}^{*} & =T_{1}^{*}-L^{*}\\
\mathcal{E}_{2} & =\varDelta-\tilde{\varDelta}.
\end{align*}
From Lemma \ref{lem:Algebra-properties-for} and equation \eqref{eq:difference with potential}
we have 
\[
\left\Vert \mathcal{E}_{1}^{*}\right\Vert _{0}=\left\Vert T_{1}^{*}-L^{*}\right\Vert _{0}\leq KG^{3/2}\left\Vert T_{1}^{*}-L^{*}\right\Vert _{1/2}\leq KG^{-5}.
\]
For the second term we use Lemmas \ref{lem:Algebra-properties-for}, \ref{lem:Tech lemma composition}
and the bounds for $\tilde{\varDelta}$ and $C$ from Theorems \ref{thm: Final Theorem sect 3}
and \ref{thm:change of variables C} to obtain
\begin{align*}
\left\Vert \mathcal{E}_{2}\right\Vert _{0} & =\left\Vert \tilde{\varDelta}\circ\left(\mathrm{Id}+C\right)-\tilde{\varDelta}\right\Vert _{0}\leq KG^{3}\left\Vert \tilde{\varDelta}\right\Vert _{0}\left\Vert C\right\Vert _{0}\\
 & \leq KG^{9/2}\left\Vert \tilde{\varDelta}\right\Vert _{1/2}\left\Vert C\right\Vert _{0}\leq KG^{-7/2},
\end{align*}

Combining these results 
\[
\left\Vert \mathcal{E}\right\Vert _{0}\leq KG^{-7/2}.
\]
Hence, by direct application of Lemma \ref{lem:Exponentially small technical lemma}
we obtain that for $u\in D_{\kappa_{3},\delta_{3}}\cap\mathbb{R}$
\[
\left|\mathcal{E}^{[l]}\left(w\right)\right|\leq G^{-7/2}K^{\left|l\right|}e^{\frac{-\left|l\right|G^{3}}{3}}.
\]
Now, defining $E=\mathcal{E}^{[0]}$ (notice that by Lemma \ref{lem:Exponentially small technical lemma}, $\mathcal{E}^{0}$ is constant) we have that for $\left(w,\xi\right)\in\left(D_{\kappa_{3},\delta_{3}}\cap\mathbb{R}\right)\times\mathbb{T}_{\sigma_{3}}$
\begin{align*}
\left|\mathcal{E}-E\right| & \leq\sum_{\left|l\right|>1}\left|\mathcal{E}^{[l]}\left(w\right)\right|\\
 & \leq G^{-\nicefrac{7}{2}}e^{\frac{-G^{3}}{3}}\sum_{\left|l\right|>2}\left(Ke^{\frac{-G^{3}}{3}}\right)^{\left|l\right|}\\
 & \leq KG^{-\nicefrac{7}{2}}e^{\frac{-G^{3}}{3}}.
\end{align*}
Finally, it is a straightforward computation to check that 
\[
\left|\frac{\text{d}}{\text{d}w}\mathcal{E}^{[l]}\left(w\right)\right|\leq G^{-1/2}K^{\left|l\right|}e^{\frac{-\left|l\right|G^{3}}{3}}
\]
so we conclude that 
\[
\left|\partial_{w}\mathcal{E}\right|\leq KG^{-\nicefrac{1}{2}}e^{\frac{-G^{3}}{3}}.
\]
 
\end{proof}
There is only one step left for achieving our goal, going back to
the original variables $\left(v,\xi\right)$. This is done in the
next lemma.
\begin{lem}
Consider the function
\[
\tilde{\mathcal{E}}\left(v,\xi\right)=\tilde{\varDelta}\left(v,\xi\right)-L\left(v,\xi\right)
\]
where $\tilde{\varDelta}\left(v,\xi\right)$ is defined in \eqref{eq:Difference between manifolds def.}
and $L\left(v,\xi\right)$ is defined in \eqref{eq:Melnikov potential}.
Fix $\kappa_{4}>\kappa_{3},$ $\delta_{4}<\delta_{3}$ and $\sigma_{4}>\sigma_{3}$.
Then, for $\left(v,\xi\right)\in\left(D_{\kappa_{4},\delta_{4}}\cap\mathbb{R}\right)\times\mathbb{T}_{\sigma_{4}},$
\begin{equation}
\left|\tilde{\mathcal{E}}\left(v,\xi\right)-E\right|\leq KG^{-\nicefrac{7}{2}}e^{\frac{-G^{3}}{3}}\label{eq:difference}
\end{equation}
where $E$ is a constant and 
\begin{equation}
\left|\partial_{v}\tilde{\mathcal{E}}\left(v,\xi\right)\right|\leq KG^{\nicefrac{-1}{2}}e^{\frac{-G^{3}}{3}}.\label{eq:difference derivative}
\end{equation}
\end{lem}

\begin{proof}
We look for a function $\varphi\left(v,\xi\right)$ such that $\left(\mathrm{Id}+C\right)\circ\left(\text{\ensuremath{\mathrm{Id}}}+\varphi\right)\left(v,\xi\right)=\left(v,\xi\right)$,
i.e., $\varphi$ must satisfy 
\[
v=v+\varphi\left(v,\xi\right)+C\left(v+\varphi\left(v,\xi\right),\xi\right)
\]
or what is the same 
\begin{equation}
\varphi\left(v,\xi\right)=-C\left(v+\varphi\left(v,\xi\right),\xi\right).\label{eq:final change of vars}
\end{equation}

In order to solve this fixed point equation we first use Lemma \ref{lem:Tech lemma composition}
to obtain that for $\text{\ensuremath{\varphi\in}}B\left(KG^{-4}\right)\subset\mathcal{Y}_{0,\kappa_{4},\delta_{4},\sigma_{4}}$
\[
\left\Vert C\circ\left(\mathrm{Id}+\varphi\right)\right\Vert _{0}\leq\left\Vert C\right\Vert _{0}\leq KG^{-4}
\]
so the map $\varphi\mapsto C\circ\left(\mathrm{Id}+\varphi\right)$
is well defined from $B\left(KG^{-4}\right)\subset\mathcal{Y}_{0,\kappa_{4},\delta_{4},\sigma_{4}}\rightarrow\mathcal{Y}_{0,\kappa_{4},\delta_{4},\sigma_{4}}$.
Moreover we get that there exists $b_{6}$ such that 
\[
\left\Vert C\circ\left(\mathrm{Id}+\varphi\right)_{|\varphi=0}\right\Vert _{0}\leq\frac{b_{6}}{2}G^{-4}.
\]
Since for $\varphi_{1},\varphi_{2}\in B\left(KG^{-4}\right)\subset\mathcal{Y}_{0,\kappa_{4},\delta_{4},\sigma_{4}}$
we have 
\[
\left\Vert C\circ\left(\mathrm{Id}+\varphi_{2}\right)-C\circ\left(\mathrm{Id}+\varphi_{1}\right)\right\Vert _{0}\leq KG^{-4}\left\Vert \varphi_{2}-\varphi_{1}\right\Vert _{0}
\]
we have shown the existence of a unique $\varphi\in B\left(b_{6}G^{-4}\right)\subset\mathcal{Y}_{0,\kappa_{4},\delta_{4},\sigma_{4}}$
solving \eqref{eq:final change of vars}.

Now that we have obtained the inverse change of variables, the bounds
\eqref{eq:difference} and \eqref{eq:difference derivative}
follow from direct application of Lemma \ref{lem:Tech lemma composition}
if we notice that 
\begin{align*}
\mathcal{E}\left(w\left(v,\xi\right),\xi\right) & =\left(\varDelta-L\right)\circ\left(\mathrm{Id}+\varphi\right)\left(v,\xi\right)\\
 & =\left(\tilde{\varDelta}\circ\left(\mathrm{Id}+C\right)-L\right)\circ\left(\mathrm{Id}+\varphi\right)\left(v,\xi\right)\\
 & =\tilde{\varDelta}\left(v,\xi\right)-L\circ\left(\mathrm{Id}+\varphi\right)\left(v,\xi\right)
\end{align*}
so 
\[
\tilde{\mathcal{E}}\left(v,\xi\right)=\mathcal{E}\left(v,\xi\right)+L\circ\left(\mathrm{Id}+\varphi\right)\left(v,\xi\right)-L\left(v,\xi\right).
\]
Then, the result follows from Lemma \ref{lem:Tech lemma composition}
and the estimates on Proposition \ref{prop:Melnikov potential}. 
\end{proof}

\section{\label{sec:Computation-of-the}Computation of the melnikov potential}

We devote this section to the computation of the Melnikov potential
$L\left(v,\xi\right)$ whose partial derivative with respect to $v$
gives us the first order term of the distance between the infinity
manifolds. From its definition \eqref{eq:Melnikov potential} we have
\begin{align*}
L\left(v,\xi\right) & =\int_{-\infty}^{\infty}V\left(\tilde{r}_{h}\left(v+s\right),\xi+G^{3}s\right)\text{d}s\\
 & =\int_{-\infty}^{\infty}V\left(\tilde{r}_{h}\left(s\right),\xi+G^{3}\left(s-v\right)\right)\text{d}s.
\end{align*}
Expanding in Taylor series the square root in \eqref{eq:potential_main_order_term}
we obtain that
\[
V\left(\tilde{r}_{h}\left(s\right),\xi+G^{3}\left(s-v\right)\right)=-\sum_{k=1}^{\infty}\left(\begin{array}{c}
\frac{-1}{2}\\
k
\end{array}\right)\left(4G^{4}\right)^{-k}\int_{-\infty}^{\infty}\frac{\rho^{2k}\left(\xi+G^{3}\left(s-v\right)\right)\text{d}s}{\tilde{r}_{h}^{2k+1}\left(s\right)}.
\]
 Hence, expanding now the terms $\rho^{2k}$ in Fourier series we
get 
\[
L\left(v,\xi\right)=-\sum_{l\in\mathbb{Z}}e^{il\left(\xi-G^{3}v\right)}\sum_{k=1}^{\infty}\left(\begin{array}{c}
\frac{-1}{2}\\
k
\end{array}\right)a_{l,k}\left(4G^{4}\right)^{-k}\int_{-\infty}^{\infty}\frac{e^{ilG^{3}s}\text{d}s}{\tilde{r}_{h}^{2k+1}\left(s\right)},
\]
where 
\[
a_{l,k}=\frac{1}{2\pi}\int_{0}^{2\pi}\rho^{2k}\left(\sigma\right)e^{-il\sigma}\text{d}\sigma.
\]
Since for all $\sigma\in\left[0,2\pi\right]$ we have $\left|\rho\right|<2$ we easily bound 

\begin{equation}\label{eq: Fourier coefficients for rho}
\left|a_{l,k}\right|\leq4^{k}.
\end{equation}

Moreover, changing the integration variable to the eccentric anomaly $E$ defined
by $t=E-\sin E$ 
\[
\rho\left(E\right)=1-\cos E,
\]
we obtain that
\begin{equation}
a_{1,1}=-2J_{1}\left(1\right)\neq0
\end{equation}
where $J_{1}$ is the Bessel function of first kind.

Under the time reparametrization 
\[
s=\frac{1}{2}\left(\tau+\frac{\tau^{3}}{3}\right),
\]
we can write
\begin{equation}
\begin{split}L\left(v,\xi\right) & =-2\sum_{l\in\mathbb{Z}}e^{il\left(\xi-G^{3}v\right)}\sum_{k=1}^{\infty}\left(\begin{array}{c}
\frac{-1}{2}\\
k
\end{array}\right)a_{l,k}G^{4k}\int_{-\infty}^{\infty}\frac{e^{ilG^{3}\left(\tau+\frac{\tau^{3}}{3}\right)/2}\text{d}\tau}{\left(\tau-i\right)^{2k}\left(\tau+i\right)^{2k}}\\
 & =-2\sum_{l\in\mathbb{Z}}e^{il\left(\xi-G^{3}v\right)}\sum_{k=1}^{\infty}\left(\begin{array}{c}
\frac{-1}{2}\\
k
\end{array}\right)a_{l,k}G^{4k}I\left(l,k\right)\\
 & =\sum_{l\in\mathbb{Z}}L^{[l]}e^{il\left(\xi-G^{3}v\right)}.
\end{split}
\label{eq:Melnikov expansion}
\end{equation}
The harmonic with $l=0$ is readily bounded using that 
\[
I\left(0,k\right)=\sqrt{\pi}\frac{\varGamma\left(2k-1/2\right)}{\varGamma\left(2k\right)},
\]
where $\varGamma$ stands for the Gamma function.

A standard computation shows that $L^{[l]}=L^{[-l]}$ so we focus
only on the case $l>0.$ The next proposition, which can be deduced
from Propositions 19 and 22 in \cite{Delshams2019} gives estimates
for $\left|I\left(l,k\right)\right|$ and the asymptotic first order
term for $I\left(1,1\right)$ which we use to identify the main term
in $L^{[1]}\left(v,\xi\right).$
\begin{prop}
\label{prop:Integrales Tere y Amadeu}Let $G$ be large enough, then
the estimate 
\[
\left|I\left(l,k\right)\right|\leq8e^{l}G^{3k-3/2}e^{\frac{-lG^{3}}{3}},
\]
holds for $l\geq1,$ $k\geq1$. Moreover we have that 
\[
I\left(1,1\right)=\sqrt{\pi}\left(\frac{G}{2}\right)^{3/2}e^{\frac{-G^{3}}{3}}\left(1+\mathcal{O}\left(G^{-3/2}\right)\right).
\]
\end{prop}

For $l=1$ we have 
\[
L^{[1]}=-2\left(-\frac{1}{2}a_{1,1}G^{-4}I_{1,1}+\sum_{k=2}^{\infty}\left(\begin{array}{c}
\frac{-1}{2}\\
k
\end{array}\right)a_{1,k}G^{4k}I\left(1,k\right)\right).
\]
Using Proposition \ref{prop:Integrales Tere y Amadeu} and the estimate
in \eqref{eq: Fourier coefficients for rho} we have that 
\begin{align*}
\left|\sum_{k=2}^{\infty}\left(\begin{array}{c}
\frac{-1}{2}\\
k
\end{array}\right)a_{1,k}G^{4k}I\left(1,k\right)\right| & \leq8e^{1/2}e^{\frac{-G^{3}}{3}}G^{-\nicefrac{3}{2}}\sum_{k=2}^{\infty}G^{-k}\\
 & \leq16e^{1/2}e^{\frac{-G^{3}}{3}}G^{-\nicefrac{7}{2}}.
\end{align*}
Therefore 
\[
L^{[1]}=a_{1,1}\sqrt{\pi}2^{-\nicefrac{3}{2}}G^{-\nicefrac{5}{2}}e^{\frac{-G^{3}}{3}}\left(1+\mathcal{O}\left(G^{-1}\right)\right).
\]

For $l\geq2$ we have 
\[
L^{[l]}=-2\sum_{k=1}^{\infty}\left(\begin{array}{c}
\frac{-1}{2}\\
k
\end{array}\right)a_{l,k}G^{-4k}I_{l,k}
\]
and again from Proposition \ref{prop:Integrales Tere y Amadeu} and
the estimate in \eqref{eq: Fourier coefficients for rho} we obtain
\[
\left|L^{[l]}\right|\leq32e^{l-\nicefrac{1}{2}}G^{-\nicefrac{5}{2}}e^{\frac{-lG^{3}}{3}}.
\]

From the estimates we have obtained for $\left|L^{[l]}\right|$ the
double series is absolutely convergent, which justify the expansions
in Taylor and Fouier series and the proof of Proposition \ref{prop:Melnikov potential}
is completed.

\section*{Acknowledgements}
M. G. and J. P. have received funding from the Euuropean Research Council (ERC) under the European Union{'}s Horizon 2020 research and innovation programme (grant agreement No 757802).
T. M. S. has been also partly supported by
the Spanish MINECO-FEDER Grant PGC2018-098676-B-100 (AEI/FEDER/UE) and the Catalan grant 2017SGR1049. M. G. and T. M. S. are supported by the Catalan Institution for Research
and Advanced Studies via an ICREA Academia Prize 2019. C. Vidal is supported by Fondecyt 1180288.

\bibliography{biblio}

\providecommand{\bysame}{\leavevmode\hbox to3em{\hrulefill}\thinspace}
\providecommand{\MR}{\relax\ifhmode\unskip\space\fi MR }
\providecommand{\MRhref}[2]{%
  \href{http://www.ams.org/mathscinet-getitem?mr=#1}{#2}
}
\providecommand{\href}[2]{#2}
\begin{thebibliography}{DKdlRS19}

\bibitem[AKN88]{Arnold1988}
Vladimir~I Arnold, VV~Kozlov, and AI~Neishtadt, \emph{Dynamical systems iii},
  Springer, 1988.

\bibitem[BDDV17]{Brandao2017}
L{\'u}cia Brand{\~a}o~Dias, Joaqu{\'\i}n Delgado, and Claudio Vidal,
  \emph{Dynamics and chaos in the elliptic isosceles restricted three-body
  problem with collision}, Journal of Dynamics and Differential Equations
  \textbf{29} (2017), no.~1, 259--288.

\bibitem[BDV08]{Brandao2008}
L{\'u}cia Brand{\~a}o~Dias and Claudio Vidal, \emph{Periodic solutions of the
  elliptic isosceles restricted three-body problem with collision}, Journal of
  Dynamics and Differential Equations \textbf{20} (2008), no.~2, 377--423.

\bibitem[BF04]{Baldoma2004a}
Inmaculada Baldom{\'a} and Ernest Fontich, \emph{Stable manifolds associated to
  fixed points with linear part equal to identity}, Journal of Differential
  Equations \textbf{197} (2004), no.~1, 45--72.

\bibitem[BFGS12]{Baldoma2012}
Inmaculada Baldom{\'a}, Ernest Fontich, Marcel Guardia, and Tere~M Seara,
  \emph{Exponentially small splitting of separatrices beyond {M}elnikov
  analysis: rigorous results}, Journal of Differential Equations \textbf{253}
  (2012), no.~12, 3304--3439.

\bibitem[BFM20a]{Baldoma2020a}
Inmaculada Baldom{\'a}, Ernest Fontich, and Pau Mart{\'\i}n, \emph{Invariant
  manifolds of parabolic fixed points (i). {E}xistence and dependence on
  parameters}, Journal of Differential Equations \textbf{268} (2020), no.~9,
  5516--5573.

\bibitem[BFM20b]{Baldoma2020b}
\bysame, \emph{Invariant manifolds of parabolic fixed points (ii).
  {A}pproximations by sums of homogeneous functions}, Journal of Differential
  Equations \textbf{268} (2020), no.~9, 5574--5627.

\bibitem[DKdlRS19]{Delshams2019}
Amadeu Delshams, Vadim Kaloshin, Abraham de~la Rosa, and Tere~M Seara,
  \emph{Global instability in the restricted planar elliptic three body
  problem}, Communications in Mathematical Physics \textbf{366} (2019), no.~3,
  1173--1228.

\bibitem[DS92]{Delshams1992}
Amadeu Delshams and Tere~M Seara, \emph{An asymptotic expression for the
  splitting of separatrices of the rapidly forced pendulum}, Communications in
  {M}athematical {P}hysics \textbf{150} (1992), 433--433.

\bibitem[DS97]{Delshams1997}
A.~Delshams and T.M. Seara, \emph{Splitting of separatrices in {H}amiltonian
  systems with one and a half degrees of freedom}, Math. Phys. Electron. J.
  \textbf{3} (1997), Paper 4, 40 pp. (electronic). \MR{MR1474213 (98k:58197)}

\bibitem[Ga12]{Gaivao2012}
Jos\'{e}~Pedro Gaiv\~{a}o, \emph{Exponentially small splitting of
  separatrices}, Bol. Soc. Port. Mat. (2012), no.~Special Issue, 181--184.
  \MR{3098782}

\bibitem[GaG11]{Gaivao2011}
Jos\'{e}~Pedro Gaiv\~{a}o and Vassili Gelfreich, \emph{Splitting of
  separatrices for the {H}amiltonian-{H}opf bifurcation with the
  {S}wift-{H}ohenberg equation as an example}, Nonlinearity \textbf{24} (2011),
  no.~3, 677--698. \MR{2765480}

\bibitem[Gel97]{Gelfreich1997}
Vassili~G Gelfreich, \emph{Melnikov method and exponentially small splitting of
  separatrices}, Physica D: Nonlinear Phenomena \textbf{101} (1997), no.~3-4,
  227--248.

\bibitem[Gel00]{Gelfreich00}
V.~G. Gelfreich, \emph{Separatrix splitting for a high-frequency perturbation
  of the pendulum}, Russ. J. Math. Phys. \textbf{7} (2000), no.~1, 48--71.
  \MR{1832773}

\bibitem[GK12]{GorodetskiK12}
A.~Gorodetski and V~Kaloshin, \emph{Hausdorff dimension of oscillatory motions
  for restricted three body problems}, Preprint, available at
  {http://www.terpconnect.umd.edu/~vkaloshi}, 2012.

\bibitem[GMS16]{Guardia2016}
Marcel Guardia, Pau Mart{\'\i}n, and Tere~M Seara, \emph{Oscillatory motions
  for the restricted planar circular three body problem}, Inventiones
  {M}athematicae \textbf{203} (2016), no.~2, 417--492.

\bibitem[GOS10]{Guardia2010}
Marcel Guardia, Carme Oliv{\'e}, and Tere~M Seara, \emph{Exponentially small
  splitting for the pendulum: a classical problem revisited}, Journal of
  nonlinear science \textbf{20} (2010), no.~5, 595--685.

\bibitem[Gua12]{Guardia2012}
Marcel Guardia, \emph{Splitting of separatrices in the resonances of nearly
  integrable {H}amiltonian systems of one and a half degrees of freedom}, arXiv
  preprint arXiv:1204.2784 (2012).

\bibitem[HMS88]{Holmes1988}
Philip Holmes, Jerrold Marsden, and Jurgen Scheurle, \emph{Exponentially small
  splittings of separatrices with applications to{ KAM} theory and degenerate
  bifurcations}.

\bibitem[LS80a]{SimoL80}
J.~Llibre and C.~Sim{\'o}, \emph{Oscillatory solutions in the planar restricted
  three-body problem}, Math. Ann. \textbf{248} (1980), no.~2, 153--184.
  \MR{573346 (81f:70009)}

\bibitem[LS80b]{Llibre1980}
Jaume Llibre and Carles Sim{\'o}, \emph{Some homoclinic phenomena in the
  three-body problem}, Journal of Differential Equations \textbf{37} (1980),
  no.~3, 444--465.

\bibitem[McG73]{McGehee1973}
Richard McGehee, \emph{A stable manifold theorem for degenerate fixed points
  with applications to celestial mechanics}, Journal of Differential Equations
  \textbf{14} (1973), no.~1, 70--88.

\bibitem[Moe84]{Moeckel84}
R.~Moeckel, \emph{Heteroclinic phenomena in the isosceles three-body problem},
  SIAM Journal of Mathematical Analysis \textbf{15} (1984), no.~5, 857--876.
  \MR{86j:58047}

\bibitem[Moe07]{Moeckel07}
R.~Moeckel, \emph{Symbolic dynamics in the planar three-body problem}, Regul.
  Chaotic Dyn. \textbf{12} (2007), no.~5, 449--475. \MR{2350333}

\bibitem[Mos73]{moser1973stable}
Jurgen~K Moser, \emph{Stable and random motions in dynamical systems, volume 77
  of {A}nnals of {M}athematics studies}, 1973.

\bibitem[MP94]{Martinez1994}
Regina Mart{\'\i}nez and Conxita Pinyol, \emph{Parabolic orbits in the elliptic
  restricted three body problem}, Journal of differential equations
  \textbf{111} (1994), no.~2, 299--339.

\bibitem[Nei84]{Neishtadt1984}
Anatoly~I Neishtadt, \emph{The separation of motions in systems with rapidly
  rotating phase}, Journal of Applied Mathematics and Mechanics \textbf{48}
  (1984), no.~2, 133--139.

\bibitem[Sit60]{Sitnikov1960}
K~Sitnikov, \emph{The existence of oscillatory motions in the three-body
  problem}, Dokl. Akad. Nauk SSSR, vol. 133, 1960, pp.~303--306.

\bibitem[Tre97]{Treshev97}
D.~Treschev, \emph{Separatrix splitting for a pendulum with rapidly oscillating
  suspension point}, Russ. J. Math. Phys. \textbf{5} (1997), no.~1, 63--98.

\end{thebibliography}
\bibliographystyle{amsalpha}

\end{document}